\documentclass[11pt]{amsart}
\usepackage{amsmath,amsthm,amssymb,amscd, graphicx, wrapfig}
\usepackage[all]{xy}
\usepackage[margin=1in]{geometry}
\usepackage[usenames, dvipsnames]{color}
\usepackage{bbm,bm}
\usepackage{comment}
\usepackage[bookmarks=false,colorlinks,linkcolor=blue,citecolor=blue]{hyperref}

\numberwithin{equation}{section}

\newtheorem{theorem}{Theorem}

\newtheorem{corollary}[theorem]{Corollary}

\newtheorem{thm}[theorem]{Theorem}
\theoremstyle{definition}
\newtheorem{defn}[theorem]{Definition}
\newtheorem*{exmp}{Example}
\newtheorem*{rem}{Remark}

\def\doublestroke#1{\pdfliteral{1 Tr .325 w}#1\pdfliteral{0 Tr 0 w}}
\def\Omm{\doublestroke{\mathbf{\Omega}}}
\def\Spp{\doublestroke{\mathbf{\Psi}}}
\def\Sp{\mathbf{\Psi}}
\def\Om{\mathbf{\Omega}}
\def\sig{\mathbf{\sigma}}
\def\rw{\mathbf{\rho}}
\def\bx{\mathbf{\xi}}

\title{Compactifications of phylogenetic systems and species of electrical networks}

\author{Satyan L.\ Devadoss}
\address[S.\ Devadoss]{University of San Diego, San Diego, CA 92110}
\email{devadoss@sandiego.edu}
\urladdr{satyandevadoss.org}

\author{Stefan Forcey}
\address[S. Forcey]{University of Akron, Akron, OH 44325}
\email{sforcey@uakron.edu}  
\urladdr{www.math.uakron.edu/\~{}sf34}

\begin{document}

\begin{abstract} 
 We describe new spaces and maps. Our graphical map is a visual and numerical correspondence between spaces of circular electrical networks and circular planar split systems. When restricted to the \emph{planar} circular electrical case, this graphical map finds the split system uniquely associated with the Kalmanson resistance distance of the dual network, matching the induced split system familiar from phylogenetics.  This correspondence is extended to compactifications of the respective spaces, taking cactus networks to the cactus split systems defined herein. 
The graphical map preserves both network components and cactus structure, allowing an elegant enumeration of induced phylogenetic split systems via combinatorial species. We introduce the global spaces of circular planar electrical networks and circular split systems. These new spaces are also CW complexes, but the 0-cells of each are counted by the Bell numbers as opposed to the Catalan numbers. As species, the two sorts of global cacti are seen to be compositions in complementary ways. 
\end{abstract}

 
\baselineskip=17pt
\maketitle

%
%
\section{Introduction}

\subsection{Motivation}

 A favorite tool of phylogenetics researchers is the split system, which is a combinatorial structure associated to a metric on a collection of biological species or taxa. The split system reveals the large-scale structure of the phylogenetic tree or network. Recently it was demonstrated that the split system is also guaranteed to exist in the case of a circular planar electrical resistor network \cite{forc-pre}. Specifically, finding the split system is a consequence of a more basic result in that paper: the resistance metric for a circular planar electrical network is Kalmanson. This highlights a mathematical analogy between two problems: reconstructing a circuit in a `black box' from surface measurements, and rebuilding a genetic history from existing population data. \begin{figure}[h!]
    \centering
\includegraphics[width=.85\linewidth]{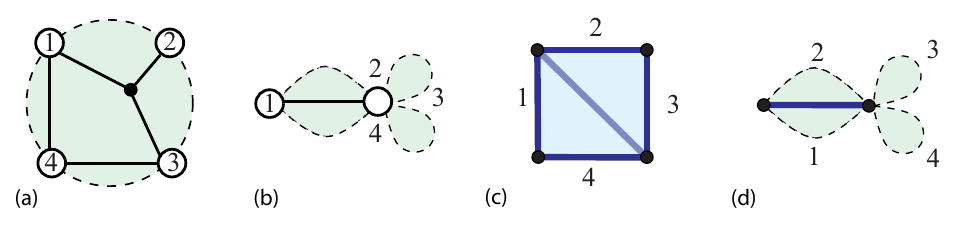}
    \caption{Example labels for the cells in the spaces we consider, for $n=4$: (a) circular planar electrical networks, (b) cactus networks  (c) circular split systems  and (d) circular cactus split systems .  The total numbers of these cells are listed in Table~\ref{topper}, row 4, and the entire sets for (a) and (c) are shown in Figure~\ref{fig:fours}.}
    \label{newone}
\end{figure}

\noindent In this paper we carry the analogy further. A printed electrical circuit is made of conductors in a network, with terminal nodes on the exterior available for connections. Some of those exterior nodes might be duplicates, short-circuited together so that they are interchangeable in terms of their use. Similarly, monozygotic twins can be seen as genetically the same for the purpose of reconstructing a family tree from the DNA samples of living individuals. On a larger scale, basic DNA sequences of extant individuals from a diverse ecosystem can be equated when they come from the same biological species. Allowing nodes, either terminals or taxa, to be identical in the limit of  being infinitely close to each other corresponds to taking an unbounded space and compactifying it. Here we define and relate the compactification of phylogenetic split systems (from the DNA examples) to the well-studied compactification of circular planar electrical network space.

We have the following applications in mind: 
\begin{enumerate} 
\item[(1)] How can a network be constructed (more quickly) to meet specified measurements of electrical properties on its boundary? 
\item[(2)] How can we determine (more quickly) from those measurements whether a network is planar?
\item[(3)] How many ways can a given circular planar electrical network be reordered at its boundary while retaining its characteristic measurements and planarity? 
\item[(4)] Can those potential reorganizations be obtained from the network's electrical properties, measured at the boundary? 
\end{enumerate}
 We insert the parenthetical ``\emph{more quickly}'' since well-known methods based on positivity answer the first two questions, as in \cite{curtisbook}. The application of such methods and our improvements is clear: a circular planar electrical network with desired responses at its terminal nodes is a circuit that can be printed on a board with terminal pins. The ability to reorder those terminal nodes while retaining the character and planarity is important to integrating the circuit into a larger design. Placing component circuits into a large array, with terminals available to connect to each other, requires first choosing the convenient ordering of the terminals on each circuit. Combining these queries, we can also make specific requests: if a circuit is planar, but we wish to add or cross wires, can the resulting circuit be made planar again?

\begin{table}[h]
    \centering
    \begin{tabular}{|c|c|c|c|c|c|c|}
    \hline
    \rule{0pt}{2.6ex}\rule[-1.2ex]{0pt}{0pt}
     $n$   & $|EP_n|$ & $|{\Omm}_n| = |P_n|$ & $|\Sp_n|$ & $|\Spp_n|$ & $|\bx(EP_n)|$  & $|\sig(\Omm_n)| = |\bx({\Omm}_n)| $ \\
     nodes &  cells of $\Om_n$ & cells of $E_n$& &&cells of $\bx({\Om}_n)$ & cells of $\bx({E}_n)$ \\
      \hline\hline
 \rule{0pt}{2.6ex}\rule[-1.2ex]{0pt}{0pt} 
      1 & 1 & 1 & 1 & 1 & 1 & 1\\
      \hline
 \rule{0pt}{2.6ex}\rule[-1.2ex]{0pt}{0pt}     
       2 & 2 & 3 & 2 & 3 & 2 & 3\\
      \hline
 \rule{0pt}{2.6ex}\rule[-1.2ex]{0pt}{0pt} 
        3 & 8 & 15 & 8 & 15 & 8 & 15\\
        \hline
 \rule{0pt}{2.6ex}\rule[-1.2ex]{0pt}{0pt}      
      4 &  52 & 105 & 64 & 117 & 49 & 102\\
      \hline
 \rule{0pt}{2.6ex}\rule[-1.2ex]{0pt}{0pt}  
      5 & 464 & 945 & 1024 & 1565 & 373 & 839\\
     \hline
 \rule{0pt}{2.6ex}\rule[-1.2ex]{0pt}{0pt}  
      OEIS & [A111088] & [A001147] & [A006125] & [A136654] & [?] & [?]\\
      \hline

    \end{tabular}
    \vspace{.15in}
     \caption{Spaces: numbers of cells in each CW complex. Left to right these are circular planar electrical networks, cactus planar electrical networks, circular split systems, and circular cactus split systems. The last two columns count the cells in the range of our maps from electrical networks to split systems. Examples are in Figure~\ref{newone}.}
    \label{topper}
\end{table}

\subsection{Background} Spaces of circular electrical networks, both planar and not planar, have been studied thoroughly over the last couple of decades. 
Collected early results are in the book by Curtis and Morrow \cite{curtisbook}. More recently, the state of the art for circular planar electrical networks is represented by Kenyon and Wilson \cite{surf, dimers, kenyon}, Kenyon and Hersh \cite{hersh1}, and Gao, Lam and Xu \cite{gao}. The compactified space of circular planar electrical networks is an embedded slice of the totally nonnegative Grassmannian, and thus projected into the \emph{amplituhedron}, as shown in several papers by Galashin, Karp, and Lam \cite{galashin, lam1, lam-ball, lam2, lam-grass}. Connections to plabic graphs (which represent classes of directed networks) have been made by Galashin, Postnikov, and Williams \cite{galashin2019higher}. Recognizing the planarity of a network and recovering its conductances both saw recent improvements in efficiency \cite{alman, kenyon}. 

 Spaces of circular split systems are better known via phylogenetics, as in the book by Steele \cite{steelphyl}, and in recent papers by Catanzaro, Gambette, Balvo\v{c}i\={u}t\.{e}, Bryant, and Spillner \cite{CATANZARO2021, Gambette2017, spillner}. Devadoss and Petti \cite{dev-petti} showed that the space of circular split systems is an intuitive extension of the Billera-Holmes-Vogtmann space of phylogenetic trees \cite{bhv}, exhibiting a natural projection into the compactification of the real moduli space of curves \cite{dev1, dev3}. It is studied as a simplicial complex by Terhorst \cite{terhorst}, and is related to the BME polytope and STSP polytope by Devadoss, Forcey, Scalzi, and Durell \cite{fpsac19, durell, scalzo}. 

In our recent papers \cite{forc-pre,  frontiers}, we demonstrated an injective mapping from electrical equivalence classes of circular planar electrical networks to circular split systems. Every planar response matrix is a Kalmanson metric, and thus associated to a unique weighted circular split system. To see if a given circular split system is the image of an electrical network is more complicated. For a given Kalmanson metric one can first reorder the matrix to match any cyclic order respected by its circular split system. Then there is a unique associated response matrix, which can be checked for nonnegative circular minors to decide if it is a planar electrical network. If so, then we show in \cite{forc-pre} that the circular planar electrical network has the same overall tree structure as the split system network of the resistance matrix. Then the algorithm from Chapter 9 of \cite{curtisbook} can reconstruct a critical subnetwork for any part not made of bridges, using the strand diagram. In \cite{moscow} the authors show how to check for planarity directly using the (reordered) resistance metric. This is done via the Lam embedding into the nonnegative Grassmannian, but with a variation that takes the resistance metric as input. Then in \cite{moscow} the authors show an alternative reconstruction of the electrical network again using the resistance metric directly to find the strand matching. 

\begin{figure}
 \includegraphics[width=\textwidth]{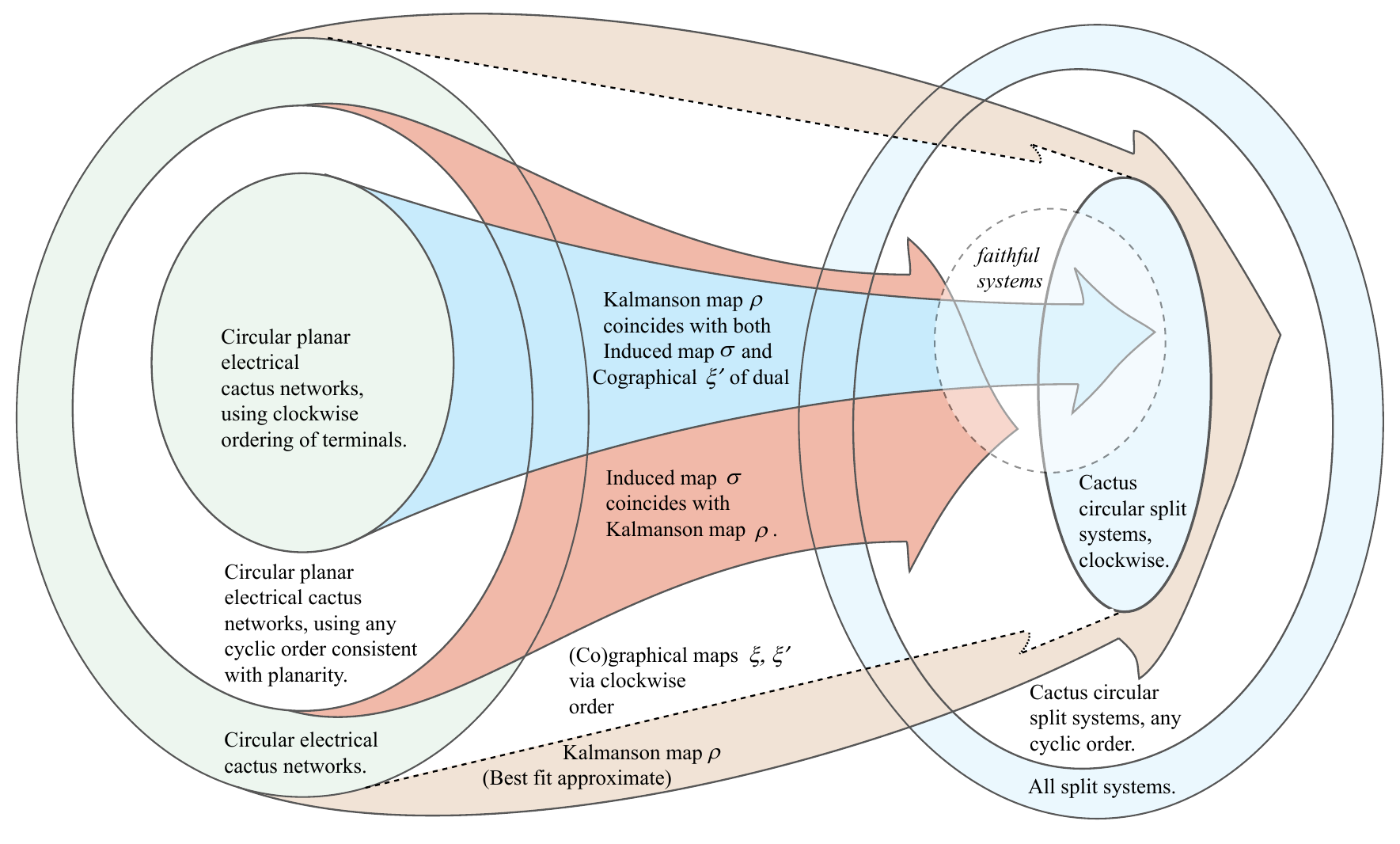}
 \caption{Domains and ranges for the maps in this paper.}
 \label{coolmap}
\end{figure}
\subsection{Overview} From a geometric combinatorial perspective, our interest lies in comparing the spaces of \textit{circular planar electrical networks} and \textit{circular split systems} as CW-complexes with analogous cell-decompositions. We will review the definitions in the following sections, but for now we show some examples of each, to introduce our notation. $\Om_n$ is the space of circular planar electrical networks with $n$ exterior nodes, in notation from \cite{curtisbook}. $\Om_n$ has its poset of cells denoted $EP_n$ in \cite{alman}. The cells correspond to equivalence classes of critical networks such as (a) in Figure~\ref{newone}. The compactification of $\Om_n$ is the space of cactus networks, denoted $E_n$ in \cite{lam}. We denote by $\Omm_n$ the poset of cells of $E_n,$  with an example in (b) of Figure~\ref{newone}.

The space of split systems on the set $[n]$ is denoted $\Sp_n.$ Subfigure~\ref{newone}~(c) shows an example of a split system in $\Sp_4.$    The compactification of $\Sp_n$ is defined in this paper and is denoted $\Spp_n$, with an example in (d). In the interest of readability, we often use one name for both the CW-complex and the poset of cells: we will refer to $|\Om_n|$ as the number of cells in $\Om_n$ and to $|\Sp_n|$ as the number of cells in $\Sp_n.$  Similarly, we will use $\Omm_n$ interchangeably with $E_n.$  

Further, we want to consider the sub-complexes of split systems that are the range of the graphical map $\bx$, which takes circular planar electrical cactus networks to circular cactus split systems. In Figure~\ref{newone} (c) is the image of (a) under $\bx$, and (d) is the image of (b). There are fewer cells in that range than split systems in general: the implied mapping on cells is not onto. Neither is it an injection of cells. However the graphical map does preserve the connected components and cactus structures, and this provides us a way to count cells in the range either in either case, whether we want just the ordinary split systems or their compactifications as well. Table~\ref{topper} shows the results of counting the cells of all dimension in the various complexes. 

The map $\bx$ is quite simple in its operation on matrices: it is the identity map (on the independent portion of the matrix.)  The identity map is clearly a homeomorphism. Thus the facts about the space $E_n$ of cactus networks as shown in \cite{lam} and \cite{lam-ball} are also true of the image of $\bx$ in the compactified space of  circular split systems. It is topologically a ball, and isomorphic to a certain subspace of the nonnegative Grassmannian. The map $\bx$ often takes several cells of $E_n$ (elements of $\Omm_n)$ and embeds those cells inside of a single cell of the space of circular cactus split systems, $\Spp_n$.  
However, the map does respect the cell structure in that, for cells, $x\le y$ in the domain iff $\bx(x) \le \bx(y)$ in the range.
Thus 
it induces a new regular CW-complex structure on the circular electrical networks.  Those new cells in $E_n$ are labeled by their Kron reductions, or by their image in $\Spp_n$ which we will designate by plabic tilings. Careful study of these cells in $E_n$ is important future work.

 In this paper we first introduce cactus split systems and extend the maps to the domain of electrical cactus networks. Then we expand to \textit{global} spaces: while traditionally the objects have a given cyclic order of $[n]$ we instead allow any cyclic order, and identify appropriate equivalent networks. That is how classical phylogentic spaces are studied, both the space of trees $BHV_n$  and the space of split networks (with all trivial splits) $CSN_n.$  The combinatorics and topology become much more complex in the global case. With the global equivalence in place even the basic enumeration of the objects becomes difficult. That is not surprising when we see how they relate to famous hard problems, specifically facets of the Symmetric Travelling Salesman polytope. Table~\ref{top2} lists the total numbers of cells in our new global spaces, both compact and non-compact. The last column there lists the number of cells in the image of the global cactus networks, under our map $\sig$ that takes a cactus network to displayed splits on the parts of the partition into its connected components.   

\begin{table}[h]
    \centering
    \begin{tabular}{|c|c|c|c|c|c|c|}
    \hline
    \rule{0pt}{2.6ex}\rule[-1.2ex]{0pt}{0pt}
     $n$   & $|\Om^{global}_n|$ & $|\Omm^{global}_n|$ & $|\Sp^{global}_n|$ & $|\Spp^{global}_n|$  & $|\sig(\Omm^{global}_n)|$ \\
      \hline\hline
 \rule{0pt}{2.6ex}\rule[-1.2ex]{0pt}{0pt} 
      1 & 1 & 1 & 1  & 1 & 1\\
      \hline
 \rule{0pt}{2.6ex}\rule[-1.2ex]{0pt}{0pt}     
       2 & 2 & 3 & 2 & 3  & 3\\
      \hline
 \rule{0pt}{2.6ex}\rule[-1.2ex]{0pt}{0pt} 
        3 & 8 & 15 & 8 & 15  & 15\\
        \hline
 \rule{0pt}{2.6ex}\rule[-1.2ex]{0pt}{0pt}      
      4 &  70 & 133 & 112 & 169  & 124\\
      \hline
 \rule{0pt}{2.6ex}\rule[-1.2ex]{0pt}{0pt}  
      5 & 1466 & 2397 & 6976 & 7857  & ?\\
     \hline
 \rule{0pt}{2.6ex}\rule[-1.2ex]{0pt}{0pt}  
      OEIS & [?] & [?] & [?] & [?]  & [?]\\
      \hline
 
     \hline

    \end{tabular}
    \vspace{.15in}
    \caption{Global spaces: numbers of cells.}
    \label{top2}
\end{table}

\subsection{Results}

We describe convenient maps between circular electrical networks and circular split systems. The latter are simpler combinatorial structures well-studied in the field of phylogenetics, so our maps can be seen as invariants of the electrical networks.  Our \textit{graphical map} $\bx$ (defined in Section~\ref{s:maps}) takes \emph{any} compactified circular network (cactus network, planar or not) to a cactus \emph{planar} split system.  The latter are defined in Section~\ref{s:cactus}. We show how to count them, with Corollary~\ref{bigform}:
 A formula for the the number $|\Spp_n|$ of unweighted cactus split systems for a given cyclic order on $[n]$ is as follows:\\
 $$|\Spp_n| =\frac{1}{n+1}\left(\sum_{j_0+\dots+j_n = n}\left(\prod_{i=0}^n |\Sp_{j_i}| \right)\right),$$\\
 where the sum is over ordered lists of non-negative integers summing to $n$, and $|\Sp_n|$ is the number of unweighted (non-compact) circular split systems: $|\Sp_n| = 2^{n \choose 2}.$

Some resulting numbers are shown in Table~\ref{topper}. The formula in Corollary~\ref{bigform} arises via Lagrange inversion from a recursive functional equation of combinatorial species composition. We use the same notation to refer to species and generating functions as we use for spaces and CW complexes. Thus for circular cactus split systems the species equation is $\Spp = \Sp \circ (X\cdot \Spp)$, as shown in Section~\ref{s:count}. The general principle we notice is that often the compactification of a space of networks is related to the original space of ordinary networks by that same functional equation.  This corollary also applies to the counting of equivalence classes of electrical cactus networks.  Since as species we have $\Omm = \Om \circ(X \cdot \Omm)$ the number $ |\Omm_n|$ of cells for cactus networks can be found by: 
$$|\Omm_n|=\frac{1}{n+1}\left(\sum_{j_0+\dots+j_n = n}\left(\prod_{i=0}^n |\Om_{j_i}| \right)\right) = (2n-1)!!,$$\\

\noindent where $|\Om_n|$ is the number of distinct classes of  non-compact electrical networks. For instance, $(5(52)+10(8)2+10(2)3+10(4) + 5(1))/5 = (2(4)-1)!!.$ The double factorial formula was shown in \cite{lam}, but the larger formula is new.\footnote{Our formula for $(2n-1)!!$ here is related to an inverse formula given by Paul D. Hanna in entry [A111088] of \cite{oeis}.} 
 
  A third instance of this principle is forthcoming. First though, if the network $N$ itself is planar, then $\bx(N)$ is a shortcut to the alternative, well-known ways to find the associated split system:

\begin{thm}\label{top}
For a planar cactus network $N$, the graphical system $\bx(N)$ coincides with both the Kalmanson and the induced split systems of the planar dual $N^*$. That is, $\bx(N) = \rw(N^*) = \sig(N^*) .$ Respectively, we have $\rw(N) = \sig(N) = \bx'(N^*).$ 
\end{thm} 

We show the domains and codomains of the maps in Figure~\ref{coolmap}. Theorem~\ref{top} is a summary of Theorems~\ref{bigth} and \ref{match} in Section~\ref{s:cactus}, as extended to the compactified case via Theorem~\ref{matchcaxt}. 
Indeed, we show that the split systems in the image of this multi-named map have a particular form, known as \textit{faithful} \cite{frontiers} --- those split systems which are made up of splits displayed by some circular planar network, phylogenetic or electrical. These faithful systems are an important subcomplex of the space of circular split systems.  Since taking duals of a circular planar split network is a bijective operation, Theorem~\ref{top} has an immediate application for counting these faithful split systems. We note that if an electrical cactus network $N$ has multiple connected components, then the faithful (induced) split system $\sig(N)$ will have cactus bulbs corresponding to connected components of $N$ and vice versa. This complicates directly counting the induced split systems. On the other hand, our graphical split system $\bx(N)$ has the same number of components and the same cactus form as the network $N$. Thus we can use our recursive counting principle to relate numbers of split systems for these faithful forms in the compact and non-compact cases, in Corollary~\ref{last} 
by the now recognizable formula: 
$$|\sig(\Omm_n)|= |\bx(\Omm_n)| = \frac{1}{n+1}\left(\sum_{j_0+\dots+j_n = n}\left(\prod_{i=0}^n |\bx(\Om_{j_i})| \right)\right).$$\\
Section~\ref{s:count} proves enumeration results (summarized by Table~\ref{topper}, page~\pageref{topper}) for our spectrum of spaces (showcased by Figure~\ref{coolmap}, page~\pageref{coolmap}). In order to recognize the form of faithful split systems we define a new class of plabic tilings for sets of polygons. Then we finish with the following (expanding Corollary~\ref{last}) :

\begin{corollary}\label{cdos}
The number of cells in the compactification of the  complex of faithful split systems with $n$ boundary nodes is
$$|\sig(\Omm_n)| \ = \ \frac{1}{n+1}\left(\ \sum_{j_0+\dots+j_n = n}\left(\ \prod_{i=0}^n \ \left(\ \sum_{s\in {\mathbf{\mathcal{T}}}_{j_i}} \ 2^{t(s)}\right) \right)\right)\,,$$\\
where $j_0,\dots,j_n$ are ordered non-negative integers, ${\mathbf{\mathcal{T}}}_{j_i}$ is the set of plabic tilings of a polygon with $j_i$ vertices, and $t(s)$ is the number of boundary edges of the shaded regions of $s$.
\end{corollary}
  
In Section~\ref{s:glob}  we introduce the globalization of the classical spaces and begin their study. We point out some initial results, such as the fact that the global spaces allow all cyclic orders, so their compactifications allow all partitions rather than just non-crossing partitions of $[n].$  In Theorem~\ref{globspecs} we see that the global cactus split systems can be described as a composition of species:  $\Spp^{global} = E_+\circ \Sp^{global}$ where $E_+$ is the species of non-empty sets. That is, a global cactus split system is formed by partitioning the set $[n]$ and then making a global ordinary split system on each  part. As exponential generating functions we have: 
\begin{large}
  $$  {\Spp^{global}(x) = e^{\Sp^{global}(x)}-1.} $$
\end{large}  
In contrast, as a species the  global cactus electrical networks are described by the opposite composition: $\Omm^{global} = \Om^{global} \circ E_+$.
As  exponential generating functions we have: 
\begin{large}
$$\Omm^{global}(x) = \Om^{global}(e^x-1).$$
\end{large}

  Section~\ref{s:corro} collects some corollaries that relate our theorems to the practical questions (1)-(4), via obstructions to circular planarity.
Using the plabic tilings we count the number of consistent cyclic orders for a planar network and connect that collection to a face of the Symmetric Traveling Salesman polytope.

%
%
\section{Classical Spaces} \label{s:spaces}
\subsection{Circular Electrical Networks}

Physically, a general electrical network $N$ is made of conducting Ohmic wires with $n$ exposed terminals that can be tested in order: We apply unit voltage to each of the terminals in turn while grounding all the remaining $n-1$ terminals.
This defines the \emph{response matrix} $M(N)$: entry $M_{ij}$ is the current at terminal $j$ of $N$ when the unit voltage is applied to terminal $i$.

A \emph{circular electrical network} is a graph
 with its selected set of $n$ \emph{terminals} (or \emph{boundary nodes}) labeled by $[n]$ and arranged on a circle. The rest of the graph lies inside the circle, and the \emph{interior nodes} are unlabeled. The edges are weighted with positive real numbers which typically stand for the conductance of each connection. A \emph{circular planar electrical network} $N$ is one that can be represented with its terminals on a bounding circle and with no crossed edges in the disk. Usually the terminals are assumed to be numbered in clockwise order. We will consider planarity with respect to any cyclic order in Section~\ref{s:glob}.

 Two circular electrical networks are \emph{electrically equivalent} if they have the same response matrix.  An equivalence class is planar if any representative is planar. 

 \begin{figure}[t]
\includegraphics[width=.9\textwidth]{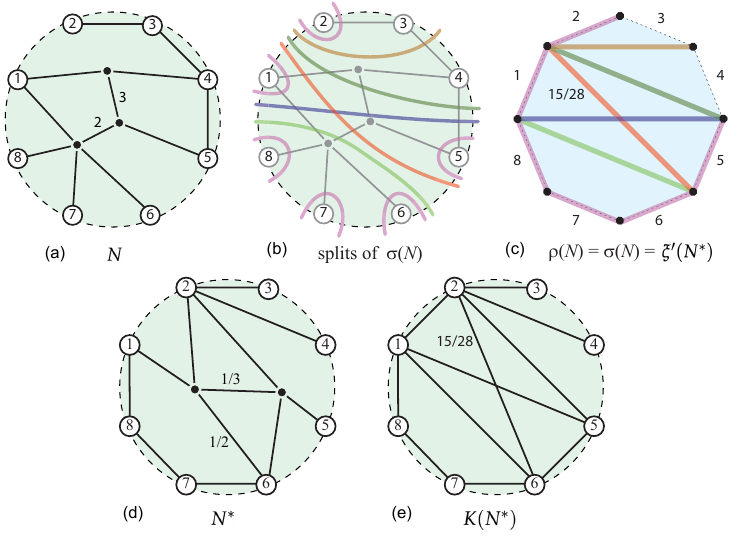}
\caption{A planar network and the related structures that we will calculate. Top row shows (a)~a clockwise planar electrical network $N$ (with conductance 1 for non-labeled edges), (b)~its splits $\sig(N)$, (c)~the polygon diagram of $\bx'(N^*) = \rho(N) = \sig(N)$, (d)~the dual $N^*$,   and (e) the Kron reduction $K(N^*)$.}
\label{f:intro}
\end{figure}

\begin{defn}
For a circular electrical network $N$, its \emph{Kron reduction} $K(N)$ is an equivalent network with the same terminal nodes as $N$, but no internal nodes. Two terminal nodes are directly connected by an edge in $K(N)$ if there is a path in $N$ connecting them which does not go through other terminal nodes. An edge $\{i,j\}$ of $K(N)$ is given the weight $M_{ij}(N)$.
\end{defn}
The Kron reduction $K(N)$ is an invariant of the electrical equivalence class of $N$. The weighted split system associated to a network is an electrical invariant as well, as shown in \cite{forc-pre}.

For circular planar networks, planar representations of two networks in the same class are related by a sequence of moves selected from: 1) replacing a series of edges, 2) replacing parallel edges, 3) deleting superfluous edges, and 4) the $Y-\Delta$ move. Those moves are pictured in \cite{curtisbook}. Also discussed in that source are several  other weaker invariants of electrical equivalence, including the set of \textit{connections} across the network, and the associated set of \textit{positive circular minors} calculated from the response matrix. All the circular minors must be nonnegative for planar networks, but which of them are positive or zero depends on the specific network. Another invariant is the sub-collection of networks of a given equivalence class with the minimal number of edges, called the \textit{critical} networks.  We do not directly use the definitions of connections and circular minors in this paper, but they are implicitly considered when we enumerate the cells of $\Om_n$. Each cell is made of response matrices that share values of these (unweighted) electrical invariants: the collections of connections, positive minors, and critical networks.

 The unweighted set of splits is also an invariant of the electrical equivalence class, as shown in \cite{forc-pre}, but like the unweighted Kron reduction, it is strictly weaker than the set of connections or the set of unweighted critical networks.
 
 Figure~\ref{f:intro}(a) shows an example of a circular planar network $N$, where the non-labeled edges have conductance 1. This is a running example, and we show how to calculate the rest of Figure~\ref{f:intro} in the following sections. The dual $N^*$ shown in part (d) is calculated in Figure~\ref{f:strands}. The conductances of the Kron reduction of the dual are given by the off-diagonal entries of the response matrix $M(N^*)$, where this matrix is computed as the Schur complement of the graph Laplacian with respect to the interior nodes (we show examples in \cite{forc-pre}). Here we only show the weight of one edge $M(N^*)_{2,6} = 15/28.$

As an alternative measurement we could use an ohmmeter to test the resistance (impedance) between pairs of our terminals. We record these results as the \emph{resistance matrix} $W(N)$, where $W_{ij}$ is the effective resistance between $i$ and $j.$ The entries $W_{ij}$ are a metric on the terminal nodes, as shown in \cite{klein93}.

\noindent Figure~\ref{network_examplo}(a) shows a non-planar case. \begin{figure}[h!]
\includegraphics[width=\textwidth]{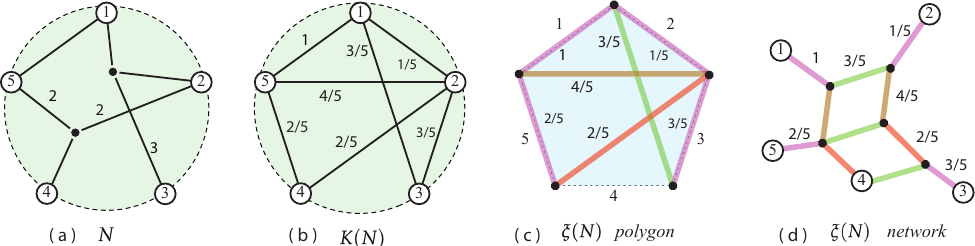}
\caption{(a)  Non-planar (with respect to the clockwise order) circular electrical network $N$.}
\label{network_examplo}
\end{figure} 
Then we show  $M(N)$ and $W(N)$ for Figure~\ref{network_examplo}(a):
$$ M = 
 \begin{bmatrix}
 -9/5&1/5&3/5&0&1\\
 1/5&-2&3/5&2/5&4/5\\
 3/5&3/5&-6/5&0&0\\
 0&2/5&0&-4/5&2/5\\
 1&4/5&0&2/5&-11/5\\
 \end{bmatrix} 
\hspace{.4in}
 W = 
 \begin{bmatrix}
 0&1&13/12&31/16&3/4\\
 1&0&13/12&23/16&3/4\\
 13/12&13/12&0&109/48&4/3\\
 31/16&23/16&109/48&0&23/16\\
 3/4&3/4&4/3&23/16&0\\
 \end{bmatrix}
$$

Figure~\ref{network_examplo}(b) displays the Kron reduction for the network in part (a). Note that $N$ and $K(N)$ have the same response matrix, and are therefore equivalent as circular electrical networks. In fact, the response matrix $M(N)$ is the Laplacian of the Kron reduction $K(N)$, and thus $K(N)$ can be viewed as a visualization associated with $M(N)$. The rest of Figure~\ref{network_examplo} will be explained in the section on maps.

Curtis and Morrow \cite{curtisbook} defined a space of response matrices $\Om_n$ as follows:

\begin{defn}
Let $\Om_n$ be the space of all response matrices for circular planar networks with $n$ distinct boundary nodes, labeled by $[n]$ in clockwise order.
\end{defn}

The space $\Om_n$ is of is stratified into cells as a CW-complex. The poset of cells is denoted $EP_n$ in \cite{alman}.  We often use the same notation for the space and the face poset of cells, determined by context. For instance the enumeration of these cells $a_n = |\Om_n| = |EP_n|$  is given recursively in \cite{alman} (and gives the OEIS sequence [A111088]): 
\begin{equation}
\label{e:main}
a_n \ = \ 2(n-1)\, a_{n-1} \ + \ \sum_{j=2}^{n-2} (j-1)\, a_j \, a_{n-j}\,, \ \ \textup{where}  \ \ a_0 =a_1 = 1, a_2 = 2 \,.
\end{equation}
 
The space $\Om_n$ has dimension ${n \choose 2}.$  Each element in a cell of dimension $k$ in $\Om_n$ can be minimally represented by choosing positive conductance weights for a (non-unique) planar  \emph{critical} (or \emph{reduced}) network with $k$ edges.  A cell $x$ is contained in another cell $y$ when the network corresponding to $x$ can be found by either deleting or contracting edges of $y$.  The left side of Figure~\ref{fig:fours} shows critical networks for each of the 52 cells of $\Om_4$. Each column corresponds to cells of $\Om_4$ of a fixed dimension, 0--6, with $f$-vector $(1, 6, 14, 16, 10, 4, 1)$.  
The $f$-vector of $\Om_5$ is $(1, 10, 40, 85, 110, 97, 65, 35, 15, 5, 1).$

\begin{figure}[h!]
\includegraphics[height=.9\textheight]{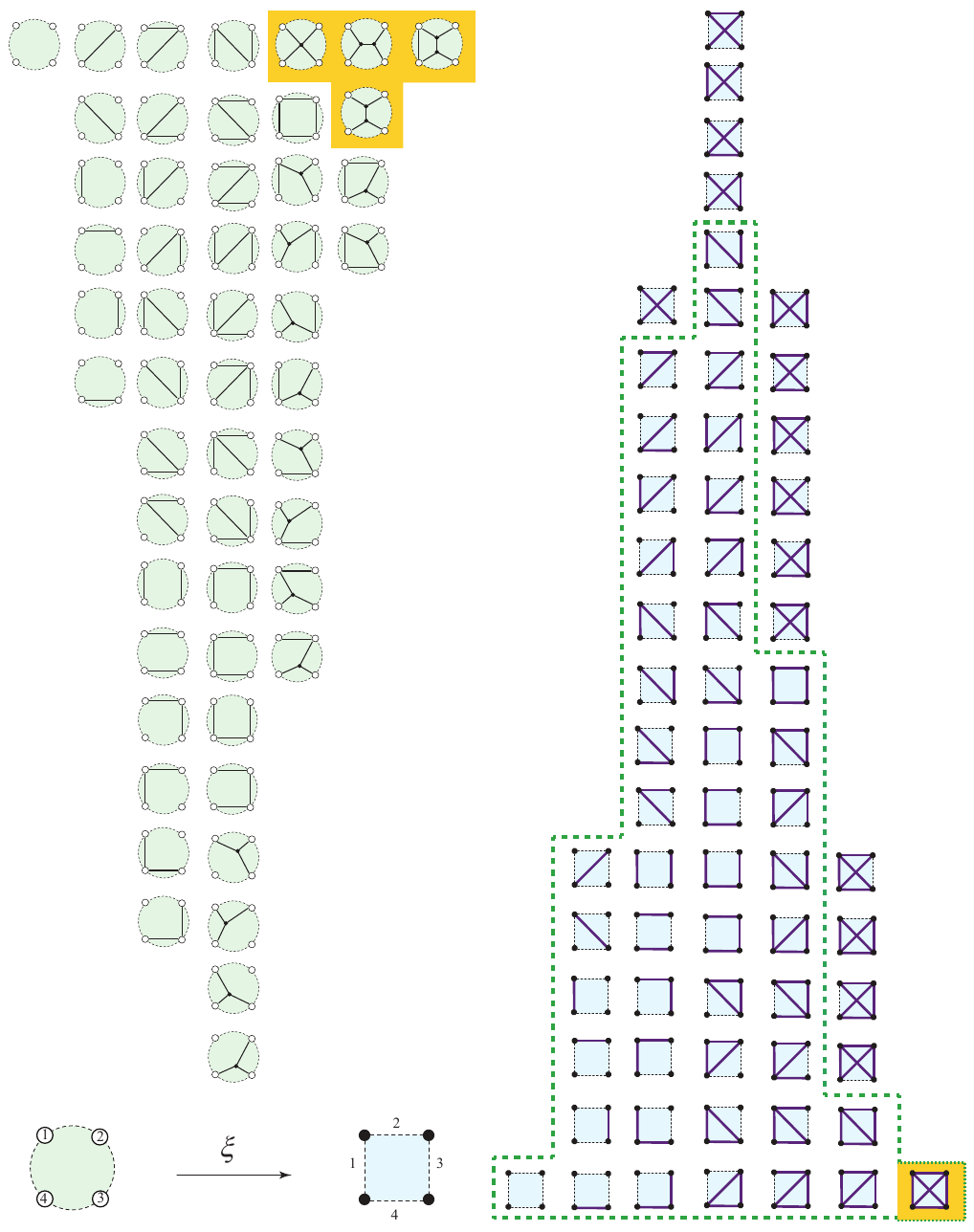}
\caption{Cells of the space $\Om_4$ on the left,  $\Sp_4$ on right. The image $\bx(\Om_4) \subset \Sp_4$ is inside the dashed line. The highlighted system on lower right is the image of the 4 networks highlighted at the top on the left.} 
\label{fig:fours}
\end{figure}

\subsection{Phylogenetic Split Systems}

A \emph{split} of $[n]$ is a bipartition $A|B$ of $[n]$ and a \emph{split system} is any collection of splits of $[n]$. 
A graph with some of its nodes labeled by $[n]$ \emph{displays} a split $A|B$ if there is a set of edges whose removal increases the number of connected components by one, and the two new components include respectively the nodes labeled by $A$ and $B$. A \emph{minimal display} of the split is a displaying set of edges that do not contain a proper subset displaying that split. 
A \emph{split network} is a representation of a split system as a graph, where each split is minimally displayed by a set of parallel edges of the same length. A \emph{circular split system} can be drawn as a split network with the $n$ \textit{boundary nodes} on the boundary of a disk, and with non-crossing sets of parallel edges.   A circular split system can also be represented by a \emph{polygonal diagram} \cite{dev-petti}: the $n$ terminals label the edges of an $n$-gon in some cyclic order, and each split is drawn as a diagonal. Figure~\ref{top} has an unweighted polygonal diagram in (c) which shows the split system on $[4]$ with all four \textit{trivial splits} (these have one singleton set)  and the split $\{1,4\}|\{2,3\}.$ Figure~\ref{network_examplo} displays a weighted polygonal diagram in (c) of the circular split network drawing in (d).
Following \cite{dev-petti} we weight the splits by assigning each split a non-negative real value independently. The weight of zero is equivalent to removing that split entirely. 

\begin{defn} 
Let $\Sp_n$ be the space of weighted circular split systems on $n$ boundary nodes, using the clockwise cyclic ordering of $[n]$. Since their are ${n \choose 2}$ splits possible, the resulting space is isomorphic to the nonnegative orthant of $\mathbb{R}^{n(n-1)/2}.$
\end{defn}

\noindent
The space $\Sp_n$ is stratified with cells corresponding to split systems with no edge weighting. The total number of cells is $|\Sp_n| = 2^{n \choose 2}$. A cell $x$ is contained in another cell $y$ when the system corresponding to $x$ can be found by removing splits of $y$.  In fact, as a CW-complex $\Sp_n$ is equivalent to a nonnegative orthant of $\mathbb{R}^{n \choose 2},$ with cells found as the regions of the coordinate hyperplanes, axes, etc. The right side of Figure~\ref{fig:fours} shows the cells of $\Sp_4$, each column corresponds to cells of a fixed dimension, with $f$-vector $( 1, 6, 15, 20, 15, 6, 1)$, the 6th row of Pascal's triangle.

Next we review, that to any circular split system $n$, we can associate a \emph{dissimilarity matrix}, an $n \times n$ real, symmetric, nonnegative matrix, where the $(ij)$-entry is the sum of the weighted splits between nodes $i$ and $j$ in the network. In our case, the dissimilarity matrix is the resistance matrix $W$. It is well-known \cite{steelphyl} that there is a bijection between weighted circular split systems and dissimilarity matrices satisfying the \emph{Kalmanson condition}: there exists a cyclic order of the boundary such that for any four nodes $i,j,k,l$ listed in that cyclic order,
\begin{equation}\label{kal}
\max(W_{ij} \ + \ W_{kl}, \ W_{jk} \ + \ W_{il}) \ \le \ W_{ik} \ + \ W_{jl}\, .
\end{equation}

Thus,  for a fixed circular labeling, the space of Kalmanson dissimilarity matrices of circular split systems on this labeling can be identified with $\Sp_n$. Kalmanson proved that the metrics obeying this condition yield solutions to the Traveling Salesman Problem in polynomial time \cite{ken} . 

%
%
\section{Maps between Classical Spaces} \label{s:maps}

In this section we explain functions that take input the connected circular (planar) electrical networks. The functions will be extended to  multiple connected components and cactus versions of the electrical networks in Section~\ref{s:cactus}, after we discuss duals on cactus networks and define cactus split systems.

\subsection{Graphical Map $\bx$}

We start by defining a key function which takes any equivalence class of circular electrical networks (planar or non-planar) as input and returns a circular \emph{planar} split system. The input circular networks (and their response matrices) must be planar with respect to a given cyclic ordering, typically the clockwise order. While all our maps are defined on weighted networks (edges weighted with conductance), and the outputs are weighted split systems, we often do abuse notation and apply the maps to unweighted input and output to see their purely combinatorial function.

\begin{defn}
The \emph{graphical split system} $\bx(N)$ of an  electrical network $N$ is the weighted circular planar split system constructed by putting the $n$ boundary nodes of $N$ in circular clockwise order and reinterpreting the Kron reduction network $K(N)$ as a polygon --- shifting its terminals by a half-step \emph{counterclockwise} rotation and labeling the exterior sides instead. The \emph{cographical split system} $\bx'(N)$ for a circular electrical network $N$ is similar but constructed with a half-step \emph{clockwise} rotation. 
\end{defn}

Here, each edge of $K(N)$ is reinterpreted as a split whose weight is equal to the conductance of the edge of $K(N).$ Since any response matrix $M$ is seen as a Kron reduction, it is clear that $\bx$ is injective and surjective, and a homeomorphism to the space of circular split systems with clockwise ordering. Indeed, seen as a function on the response matrices, it is just the identity map (on the upper triangular submatrix, which determines the rest of the response matrix), together with alternate interpretation of the effective conductance between node $i$ and $j$, the entry $M_{i,j}$.\footnote{For instance, if $n=8$, $M_{3,8}$ is the weight of the split $\{4,5,6,7,8\}|\{1,2,3\}$.} Note that this immediately gives rise to a Kalmanson metric $\mathbf{d}_{\bx}$ on $[n]$, where the distance $\mathbf{d}_{\bx}(i,j)$ between $i$ and $j$ is the sum of the splits that separate $i$ from $j.$ 

\begin{rem}
Interestingly, while we will see that this metric is the effective resistance $W(N^*)$ for the dual electric network $N^*$ for circular planar networks $N$, no meaningful interpretation of the metric $\mathbf{d}_{\bx}$ is known for general, non-planar inputs.
\end{rem}

Figure~\ref{network_examplo} shows the process for $\bx$, starting with $N$ and ending with a new circular split system $\bx(N)$. We show it in two forms, the easily seen polygonal picture in part (c) and a split network representation in part (d). Another case showing the combinatorial interpretation of $\bx'$ is seen via the running example in Figure~\ref{f:intro}, where Figure~\ref{f:strands} displays the dual network. 

\begin{corollary}
   The combinatorial map $\bx$ respects the cell structure of $\Om_n.$ That is, if $x \le y$ for cells of $\Om_n$, then $\bx(x)\le\bx(y).$
\end{corollary}
\begin{proof}
    Deleting or contracting an edge of a minimal network $x$ in $\Om_n$ can only decrease or leave constant the number of edges in the Kron reduction, that is, the splits of $\bx(x).$ 
\end{proof}
Thus the action of $\bx$ on the space of all electrical networks (all response matrices, using clockwise order) is an injective map which is surjective onto the space $\Sp_n$ of circular planar split systems in clockwise order. However the action of $\bx$ on the cell structure of $\Om_n$ is an embedding into the cell structure of $\Sp_n.$   
Figure~\ref{fig:fours} shows the combinatorial action of $\bx$ on  the cells of the space $\Om_4.$  The map $\bx$ takes all four networks (highlighted region at the top) to the single (highlighted region at the bottom) top-dimensional cell of $\Sp_4$. Even when restricted to that cell, the action of $\bx$ is not onto: only certain weighted split systems with all 6 nonzero splits arise from circular planar electrical networks. Let the values of splits be $a,b,c,d$ for the trivial splits separating 1,2,3,4 in that order, and the values be $e,f$ for the interior splits. Then the range of $\bx$ inside the set of split systems (with all  6 splits nonzero) are those satisfying $ac\ge ef$ and $bd \ge ef.$  (This follows from the non-negativity of circular minors in the response matrix.)  All the other split systems inside the dashed lines are cells for which $\bx$ is onto. The split systems not inside the dashed lines are not in the image of $\bx$ at all.

\subsection{Kalmanson Map $\rw$} 

Results in \cite{forc-pre} and \cite{frontiers} introduced the correspondence between circular planar electrical networks and circular split systems. The latter paper proved the case for level-1 networks and the former for all connected circular planar electrical networks. Here we review those results, and in Section~\ref{s:cactus} we  extend it that to the cases of cactus networks and multiple connected components. The following is proven in \cite{forc-pre}:

\begin{thm} \label{bigth}
If a symmetric matrix $M$ is a response matrix $M=M(N)$ for a connected circular planar electrical network $N$, then its resistance matrix $W(N)$ obeys the Kalmanson condition. 
\end{thm}

Theorem~\ref{bigth} implies that for any connected circular planar electrical network, there is a corresponding circular split system. The construction of that system is via the Buneman algorithm  \cite{steelphyl} or the Neighbor-Net algorithm  \cite{bm}. The former requires precise resistances, but the latter can work with approximates. For accurate data, they both give the same result, which we define as follows:

\begin{defn}
The \emph{Kalmanson split system} $\rw(N)$ is the circular split system corresponding (injectively) to the Kalmanson resistance metric $W(N)$ of the (equivalence class of the) circular planar electrical network $N$.\footnote{The Kalmanson map we call $\rw$ here is the map $R_w$ defined in \cite{forc-pre}, but we will extend it from connected planar networks to all cactus networks.} 
\end{defn}

The Kalmanson map $\rw$ can be extended to non-planar electrical networks by letting even non-Kalmanson resistance metrics be subjected to the best-fit approximation of Neighbor-net. In that extension, the map will no longer be injective. Note also that the algorithm will find a cyclic order for which $W(N)$ is Kalmanson, if possible.  This output split system will only have the original cyclic order if the original is circular planar in that order. 

\subsection{Induced Map $\sig$} 

For the case of a connected circular planar electrical network as input, we define the induced split system.
For a circular planar electrical network $N$, a \emph{grove} is a spanning forest of the graph of $N$ whose component trees each include some of the nodes labeled by $[n].$ A \emph{$k$-grove} is a grove with $k$ trees. The weight of a grove is the product of the weights of all the edges in that grove.  A 2-grove of $N$ \emph{respects} a split that is minimally displayed by $N$ if the two trees of the grove span the two components of the displayed split. 

\begin{defn}
The weight of a split displayed by a connected circular planar network $N$ is the sum of the weights of 2-groves that respect it, divided by the summed weights of the spanning trees of $N$. The \emph{induced split system} $\sig(N)$ is the set of weighted splits displayed by $N$.\footnote{Our $\sig(N)$ has the same underlying unweighted split network as given by the map $\Sigma(N)$ from Gambette \cite{scalzo, Gambette2017}.}
\end{defn}

\begin{thm}\label{t:rsigma}
For connected circular \emph{planar} electrical networks $N$, the induced split system coincides with the Kalmanson split system. That is, $\rw(N) = \sig(N)$.
\end{thm}

\begin{proof}
 Consider a connected circular planar electrical network $N$ which displays the splits $\sig(N)$. Each split corresponds to a collection of the 2-groves that Kenyon and Wilson use in their formula for the resistance $R_{ij}$: the 2-groves that are possible after deleting any set of edges that display the split \cite[Proposition 2.7]{dimers}. Note that this correspondence partitions the 2-groves. Thus the sum of weights of the splits between two nodes is the same as the sum of the weights of the 2-groves.
\end{proof}

\begin{exmp}
The weight calculation is illustrated in Figure~\ref{sp} for one of the splits of $\sig(N)$ from Figure~\ref{f:intro}(b). There are 15 spanning trees of $N$ and their weights sum to 56. For the split $\{1,6,7,8\}|\{2,3,4,5\}$, the figure shows the four ways to minimally display it, and beneath each we sum the weights of the 2-groves. The weights total to 30, and therefore the weight of the split is 30/56 = 15/28. That value matches the conductance of the edge from 2 to 6 in $K(N^*)$, as shown in Figure~\ref{f:intro}.
\end{exmp}

\begin{figure}[h]
\includegraphics[width=\textwidth]{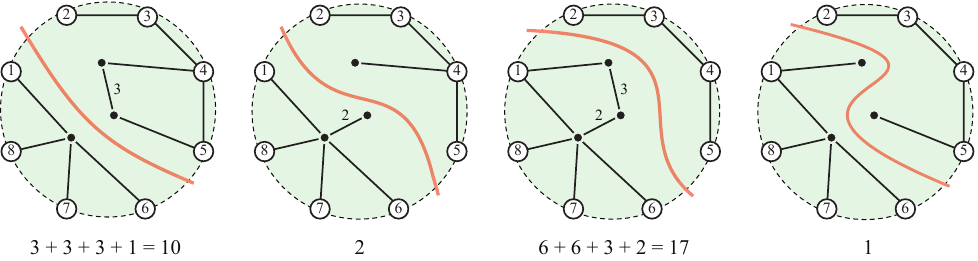}
\caption{The weight calculation for one of the splits of $\sig(N)$ from Figure~\ref{f:intro}(b).}
\label{sp}
\end{figure}

\subsection{Matchings and Duals}
We construct the \textit{dual} $N^*$ of a planar circular electrical network $N$ using \textit{strand matchings}.\footnote{Another alternative replaces each edge with a perpendicular edge; however this method is more difficult to extend to the cactus networks.}
We begin by constructing  the \textit{medial graph} of $N$.  Each original edge of $N$ becomes a new node, and any two new nodes arising from adjacent original edges of $N$ are connected by a new edge of the medial graph. Original nodes at the boundary of $N$ are always encircled by one of these new edges; and if the boundary node is degree one then the new edge will be a loop. However, we truncate the medial graph by deleting the portions that extend outside the boundary of $N$, leaving instead a pair of new nodes called \emph{stubs} on either side of each original boundary node.   

The \emph{strand diagram} of $N$ is found by tracing paths (called \emph{strands}) in the medial graph, one starting from each stub, and turning neither left nor right at new nodes, but taking the straight option to arrive eventually at another stub. 
A \emph{perfect matching} $P$ on $[2n]$ is a set of $n$ pairs $\{a,b\}$ where each element of $[2n]$ is used once. 
Every circular planar electrical network $N$ gives rise to a perfect matching on $[2n]$, by following its strands. However, this matching depends on $N$ and is not an invariant of the electrical equivalence class. To get an invariant matching, we need to restrict to the critical (reduced) representatives of that class as in \cite{lam1}, which can be recognized from the fact that strands never cross each other more than once (lens-free diagrams). In that source Lam shows that the critical representatives are in bijection with all $(2n-1)!!$ perfect matchings. Here we use the strands of even non-critical networks (as in our running example from Figure~\ref{f:intro}) to find the dual, and point out that Kron reductions of $N$ and its dual are invariants of the electrical equivalence classes. (So our maps' dependence only on the Kron reduction proves, for instance, that the calculation of the weight of a split will yield the same result no matter what specific $N$ represents the equivalence class.) 

To form the dual  $N^*$ we shift each original node of $N$ counterclockwise on the boundary just past the stub on that side. It helps visually to shade the strand diagram in a checkerboard fashion, with original nodes in shaded regions. Figure~\ref{f:strands} continues the example $N$ from Figure~\ref{f:intro}(a) along with its strand diagram $P(N$). After shifting the boundary nodes counterclockwise, they will be in unshaded regions.  Now reverse the shading and put new interior nodes in the newly shaded interior regions. Edges connect any two nodes in adjacent regions via the intersection points of the strands to complete the picture of $N^*.$  The perfect matching of $P(N)$ in Figure~\ref{f:strands} has, for instance, pairs $\{1,8\}$ and $\{2,10\}$ while that of $P(N^*)$ has pairs $\{2,9\}$ and $\{3,11\}.$ 

\begin{figure}[h]
\includegraphics[width=\textwidth]{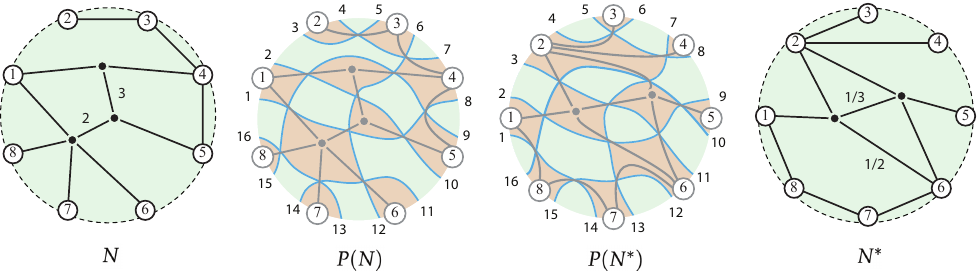}
\caption{A circular planar electrical network $N$ from Figure~\ref{f:intro} and the construction of its planar dual $N^*$ through strand diagrams.}
\label{f:strands}
\end{figure}
 
\begin{defn} \label{dual}
For a circular planar electrical network $N$, the \emph{planar dual} $N^*$ is the network that corresponds to the same strand diagram $P(N)$ but with opposite shading. Each edge of $N^*$ thus bisects an edge of $N$, where we assign the reciprocal of the edge weight to the bisecting edge of $N^*$.
\end{defn}

Although the graphical map at first seems merely a visual coincidence (Edges in the Kron reduction look like splits in a circular split system!), the next theorem shows that it actually extends our earlier maps to all circular networks, planar \emph{and} non-planar. Thus it exhibits all circular electrical networks (up to equivalence) in a one-to-one correspondence with all circular planar split systems, with the planar circular electrical networks embedded as a special subset. 

\begin{thm}\label{match}
 When restricted to connected circular \emph{planar} electrical networks $N$ with a circular clockwise order of $[n]$, the graphical map $\bx$ coincides with both the Kalmanson and the induced split systems of the planar dual. That is, $\bx(N) = \rw(N^*) = \sig(N^*).$ Respectively, we have $\bx'(N^*) = \rw(N) = \sig(N).$ 
\end{thm}

Again we prove this for the case of a connected network $N$.  \footnote{For a network with more than one connected component, the split system we obtain will have multiple components, lying in the compactification of $\Sp_n$;  see Section~\ref{s:cactus}.}

\begin{proof} 
We choose to prove the version with $\bx'(N^*)$ as it is illustrated in Figure~\ref{f:intro}. 
The second equation $\rw(N) = \sig(N)$ follows from Theorem~\ref{t:rsigma}. For the first equation, we relate $\bx'(N^*)$ and $\sig(N).$ Each split in $\sig(N)$ corresponds to the existence of at least one interior path in $N^*$.  Thus each split corresponds to an edge in the Kron reduction of $N^*$, and so a split in $\bx'(N^*)$. The weight of the edge in the Kron reduction is the same as the weight of that split, since the edge in the Kron reduction gives the effective conductance using all the interior paths.
\end{proof}


%
%
\section{Cactus Networks and Compactifications} \label{s:cactus}
\subsection{Cactii}

 A \emph{cactus electrical network} is a generalized circular planar network where boundary nodes are allowed to be identified. In particular, the \emph{space of cactus networks} is the compactification of the space of circular planar networks. In a cactus network the conductance between two boundary nodes is allowed to become $\infty$ (thereby ``\emph{short-circuiting}'' them). This is pictured by pinching together the boundary at those points, creating \textit{bulbs} of the network connected to each other by nodes that each correspond to an equivalence class of the original terminals.  The space of cactus networks was described originally in \cite{lam1}  with more examples in \cite{gao}. Figure~\ref{f:bubble}(a) shows an example along with its set of splits. Notice that in our pictures, we include \textit{empty bulbs} at each multi-identified terminal node. All the original terminal nodes $1,\dots,n$ have labels which are kept in their original clockwise order, with one bulb between each pair of nodes that are identified. This arrangement, including the empty bulbs, is important for seeing the action of our maps to the split systems.
 
 The cactus network $N$ has an associated strand diagram $P(N)$ in part (b), which can have several boundary nodes occupying one shaded region---those which are identified in the network.
Again we use the medial graph of $N$. However, to draw the strands, place all the nodes in clockwise order on a single circular boundary. Include semi-circular strands around any disconnected boundary node, and require that any set of boundary nodes all occupying the same shaded region are a set identified in the network; see part (c). Just as in Definition~\ref{dual}, the strand diagram allows the construction of the planar dual $N^*$ for a cactus network seen in Figure~\ref{f:bubble}(d). 

\begin{figure}[h]
\includegraphics[width=\textwidth]{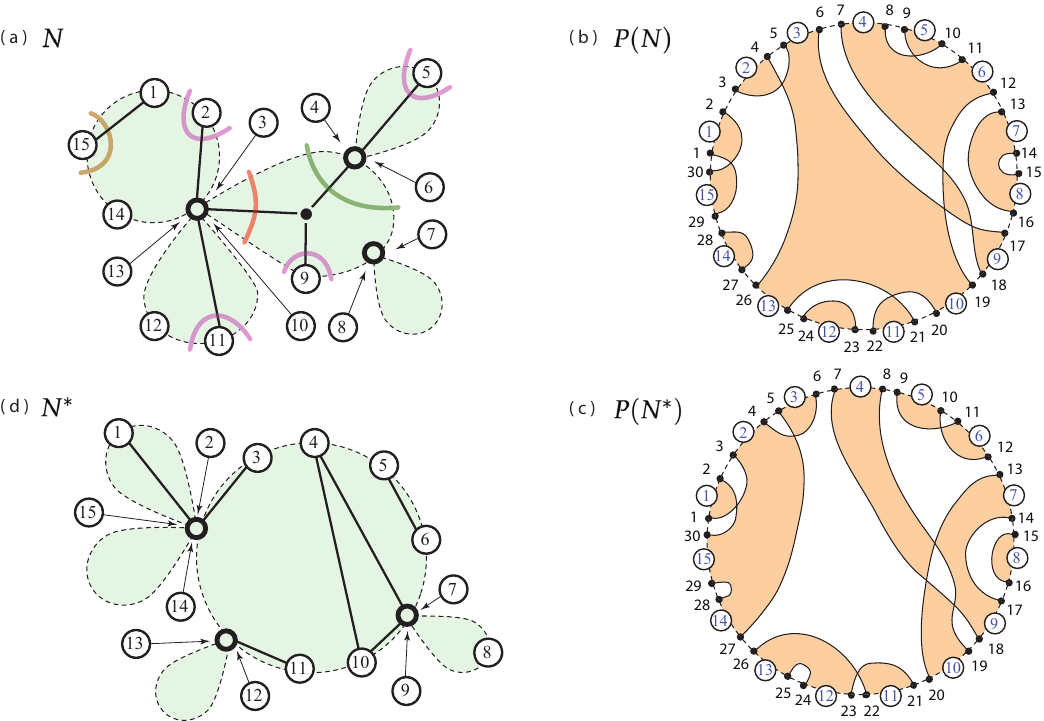}
\caption{Cactus network $N$, strand diagrams, and its dual $N^*$. The splits displayed by $N$ are shown as shaded cuts in part (a). Notice that in the dual $N^*$, the nodes comprising connected components all share the same bulb in $N.$}
\label{f:bubble}
\end{figure}

Each bulb of a cactus network is assigned its own response matrix, with rows and colums indexed by the (identified classes of) terminals on the boundary of that bulb. 
Two cactus networks are equivalent if they have the same sets of response matrices, one for each bulb. 
Lam  \cite{lam1} shows that the (unweighted) electrical equivalence classes of circular planar electrical networks, including those formed by identifying terminal points, correspond bijectively to the perfect matchings on $[2n]$.  To find the perfect matching guaranteed by this bijection, the equivalence class  must be represented by a reduced $N$.  In that case, the strands of the strand diagram will obey the requirement that any two of them cross each other at most once.

\begin{defn}
For a given cyclic order on $[n]$, the space of equivalence classes of cactus networks is the compactification ${\Omm}_n$ of the space of circular electrical networks $\Om_n$.\footnote{As mentioned earlier, in \cite{lam1}, the space ${\Omm}_n$  is denoted  as $E_n$.}
\end{defn}

The equivalence classes of  unweighted cactus networks correspond to cells of a CW-complex structure on the space. (We abuse notation by denoting this face poset as ${\Omm}_n$ as well.) Cell containment is seen by deletion or contraction of any edges.  

\begin{exmp}
Figure~\ref{uno}(a) illustrates the complex of unweighted equivalence classes of cactus networks for $n=3,$ with $f$-vector (5, 6, 3, 1). Figure~\ref{fig:skelly4Omm} shows the 1-skeleton of the complex ${\Omm}_4$ (with cells labeled by split systems). The $f$-vector for the complex  ${\Omm}_4$ is (14, 28, 28, 20, 10, 4, 4, 1). Since the face poset ${\Omm}_n$ is isomorphic to the poset $P_n$ of perfect matchings,  $|{\Omm}_n| = (2n-1)!!\,.$
\end{exmp}

\begin{figure}[h!]
\includegraphics[width=.865\textwidth]{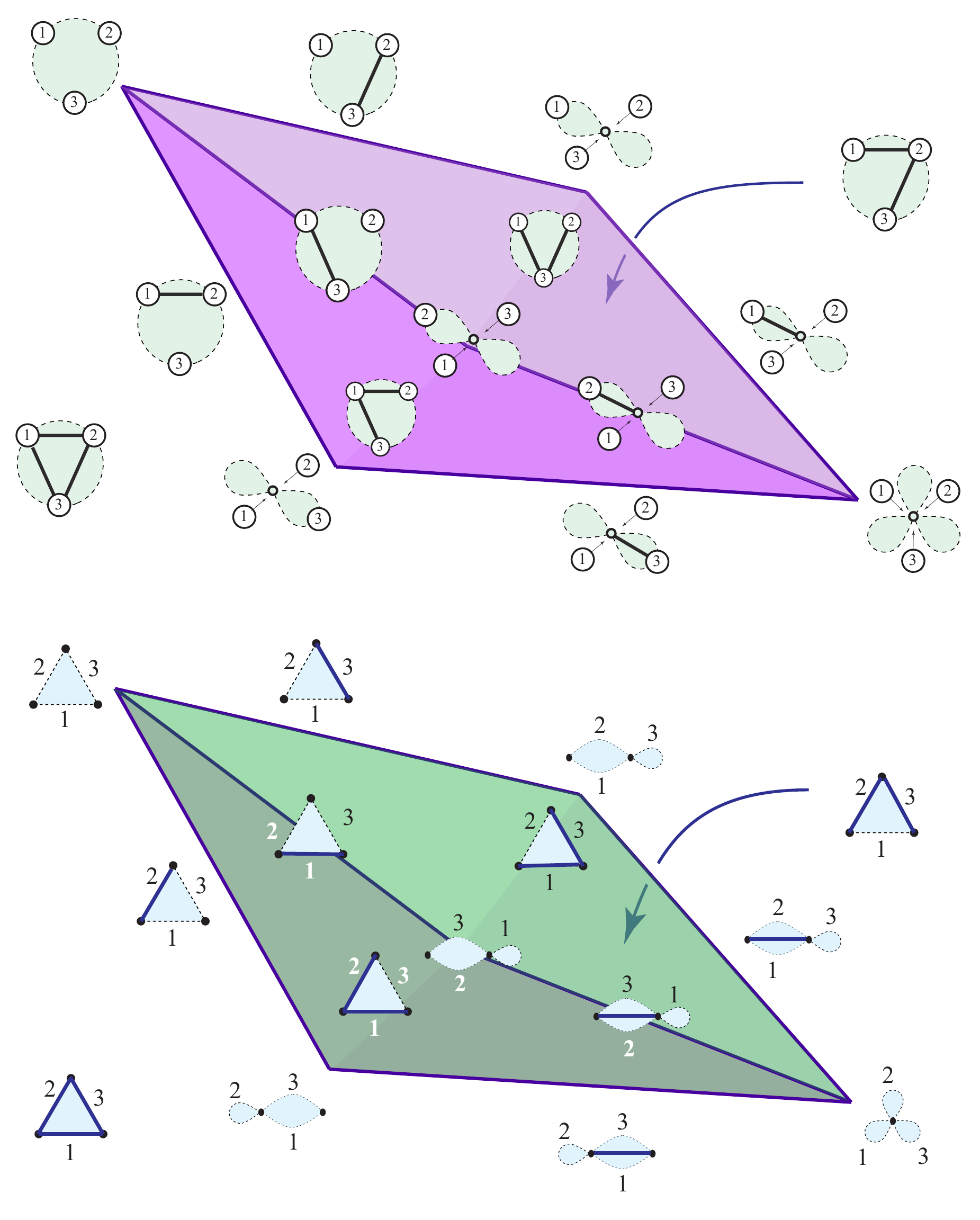} 
\caption{The graphical map $\bx$ from cactus electrical networks ${\Omm}_3$ to cactus split systems $\Spp_3$. Cells with $k$ edges (or splits) are $k$-dimensional; there are three 2-cells in each picture, one is on the back of the 3-cell.}
\label{uno}
\end{figure}

\subsection{Compactifying Split Systems}\label{commo}

Trees are a special case of both circular electrical networks and split systems. Kim \cite{kim} introduced a compactification by allowing edges of trees to become infinite in weight, resulting in a space of ``phylogenetic oranges'' \cite{ms1}. 

\begin{figure}[h]
\includegraphics{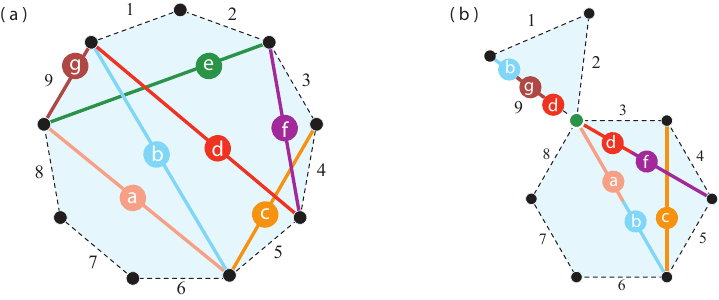}
\caption{Compactifying by weighting edge $e$ to become $\infty$.}
\label{f:poly-compact}
\end{figure}

An analogous procedure can be followed for weighted split systems. Consider the dual polygon representation of a system in Figure~\ref{f:poly-compact}(a). When we compactify this network by allowing an edge (say the diagonal `$e$' labeled in green) to become infinite in weight, this corresponds to contracting the edge to a vertex, where the ends of any edges intersecting $e$ are also contracted along with $e$; see part (b).
Note that this is dual to identification in cactus networks since resistance is the reciprocal of conductance. If this contraction results in overlapping of edges, the weights of these edges are summed. This is observed for edges labeled $g + b + d$ in Figure~\ref{f:poly-compact}(b) as well as $a+b$ and $d+f$.

Thus, a  \emph{cactus split system} for a given cyclic order (clockwise) of $[n]$ is a \emph{non-crossing partition} on $[n]$ and a weighted split system on each part of that partition. Figure~\ref{f:cactus-poly} shows the associated cactus split system to the network $N^*$ from Figure~\ref{f:bubble}, with splits appropriately color-coded.

\begin{defn}
The space $\Spp_n$ is the set of weighted circular cactus split systems with clockwise cyclic order of $[n]$. Thus it is the set of Kalmanson metrics making up the blocks of $n \times n$ dissimilarity matrices $W$, where blocks correspond to the parts of the non-crossing partition of $[n]$. 
\end{defn}

The unweighted cactus split systems on $[n]$ correspond to the cells of a CW-complex structure on $\Spp_n$. Containment of cells corresponds to either deletion (weight = 0) or contraction (weight = $\infty$) of splits. The dimension of the top-dimensional cells of $\Spp_n$ is ${n \choose 2}$, since that is the maximum number of splits compatible with a single cyclic order. 

\begin{exmp}
Figure~\ref{uno}(b) shows the CW-complex structure of the space $\Spp_3$, whereas   Figure~\ref{fig:skelly4} shows the 1-skeleton of $\Spp_4$. The $f$-vector for $\Spp_4$ is (14, 28, 29, 24, 15, 6, 1) with 117 total cells.  The 0-cells are counted by the Catalan numbers, since they correspond to the non-crossing partitions. Formulas for the total numbers of cells are given by Theorem~\ref{countery} and its corollaries.
\end{exmp}

\begin{figure}[h]
\includegraphics[height=.9\textheight]{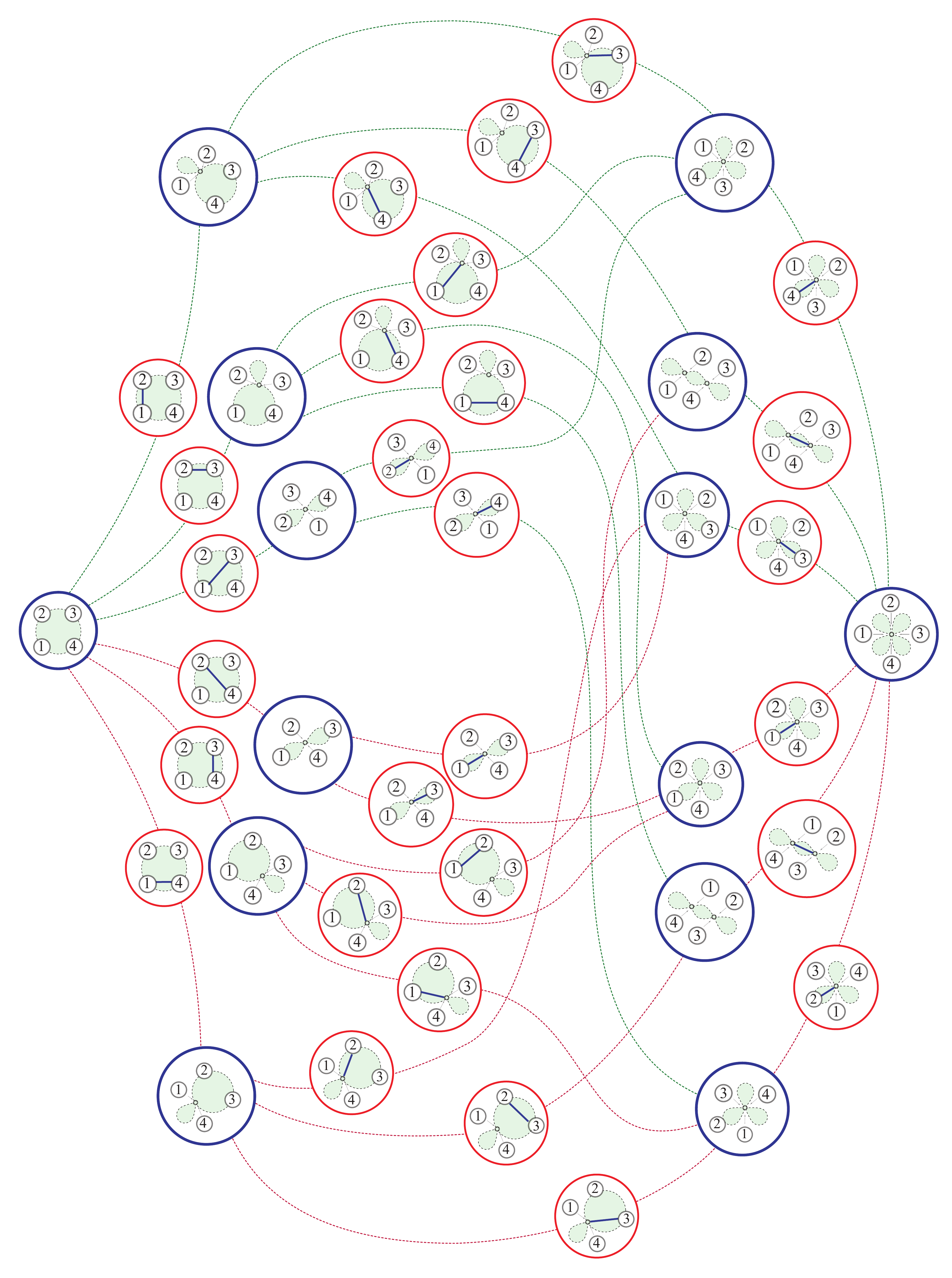}
\caption{The 1-skeleton of $\Omm_4$. The 0-cells have dark blue borders.}
\label{fig:skelly4Omm}
\end{figure}

\begin{figure}[h]
\includegraphics[height=.9\textheight]{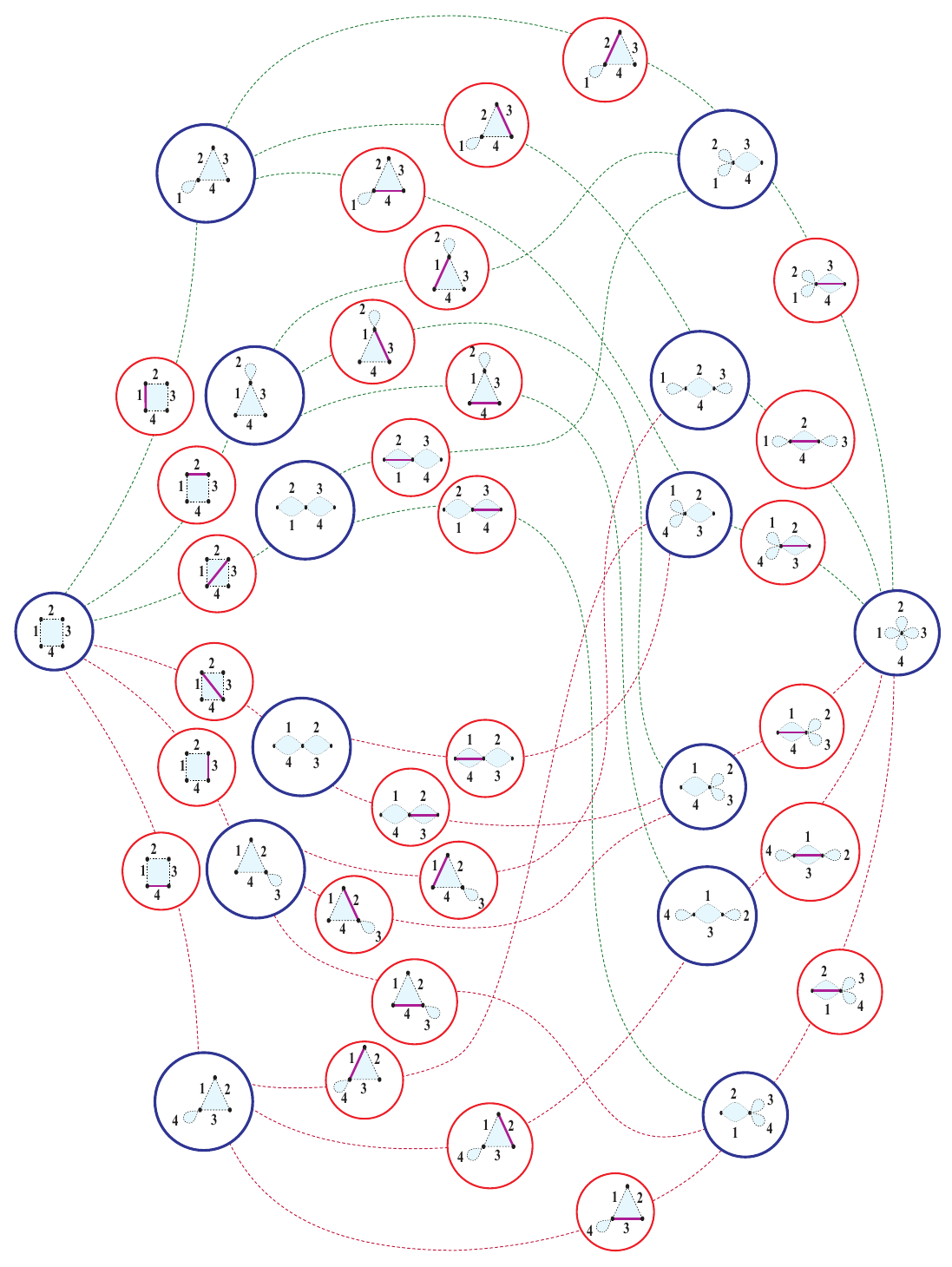}
\caption{The 1-skeleton of $\Spp_4$. Cells are images of $\bx'$ from Figure~\ref{fig:skelly4Omm}.}
\label{fig:skelly4}
\end{figure}

\subsection{Maps} 

Since resistance of $\infty$ corresponds to conductance of 0, the compactification of split systems will correspond naturally to cactus electrical networks. We extend the maps defined in Section~\ref{s:maps} to take cactus networks as input. The process is largely straightforward, but we need to carefully define the Kron reduction and response/resistance matrices of a cactus network in order to see the operation of $\bx$ and $\rw$. 

The Kron reduction $K(N)$ of a cactus network can be performed one bulb at a time. This is the same procedure as considering the response matrices of the decomposition of $N$ into one network for each bulb, embedded in its own disk, as in the proof of Theorem 4.9 in \cite{lam}.

\begin{exmp}
Figure~\ref{fig:compmap} shows a cactus network $N$ and its dual $N^*$ together with their associated split systems. Each bulb of $N$ is its own Kron reduction. Indexed by the nodes in each bulb, but allowing the matrix entries for identified nodes to coincide, we have:

$$ M(N)_{1,2,5,6} = 
 \begin{bmatrix}
 -1 & 1\\
 1 &-1\\
 \end{bmatrix} \text{~~}\text{~~~and~~~}\text{~~~} M(N)_{2,3,4,5,6} = 
 \begin{bmatrix}
 -2 & 1 &1\\
 1 &-2 &1\\
 1 &1 &-2\\
 \end{bmatrix}. $$
\end{exmp}

\begin{figure}
\includegraphics[width = \textwidth]{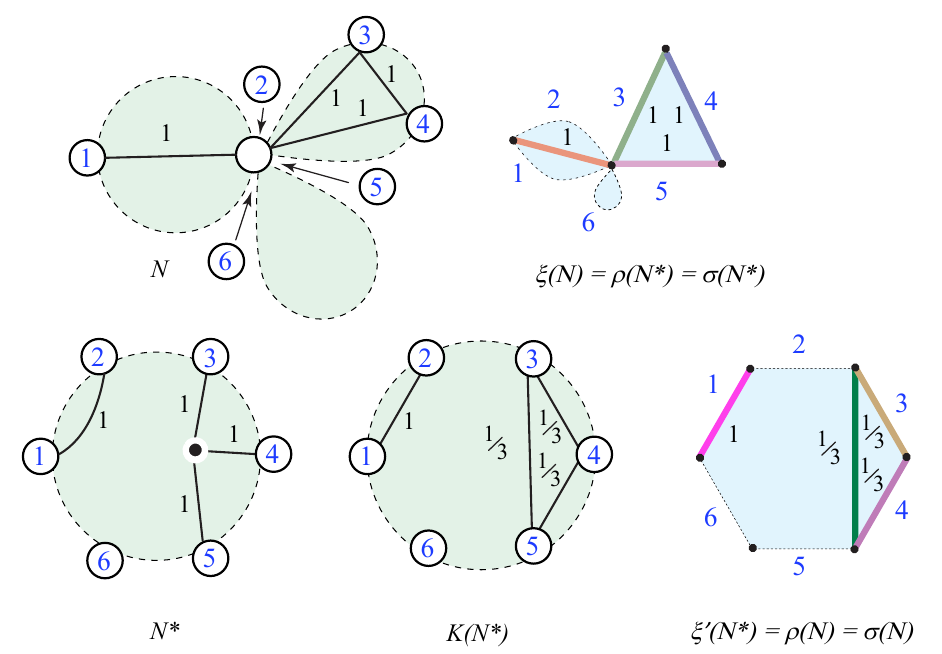}
\caption{Cactus networks, duals, Kron reductions and corresponding split systems.}
\label{fig:compmap}
\end{figure}

\subsubsection{Graphical} The graphical map $\bx$ operates on a cactus network $N$ by first finding the bulb-wise Kron reduction $K(N)$. Then each terminal node label is shifted counterclockwise to become attached instead to the nearest arc of the boundary in that direction. There is a unique nearest arc for each label since we have included empty bulbs. Then the edges of $K(N)$ are reinterpreted as weighted splits of the new cactus split system $\bx(N).$
There is also a clockwise version, the cographical map $\bx'.$

\subsubsection{Induced} The induced map $\sig$ operates on a cactus network by ignoring the bulb structure and instead acting component-wise. Each connected component displays splits of its set of terminal nodes, counting all the identified node labels as  separate elements of that set. Thus each connected component of $N$ will become a bulb of the cactus split system $\sig(N).$ That is, the connected components of $N$ give us a  non-crossing partition of $[n]$ with parts the boundary nodes of each component; this is the partition for $\sig(N)$. The splits of $\sig(N)$ on each part are those displayed in that component, with weights calculated via the spanning trees and 2-groves of that component using the edge weights (conductances). The multiple labels on nodes make no difference to the weight of a grove. 

\subsubsection{Kalmanson} The Kalmanson map $\rw$ also operates component-wise. Each connected component of $N$ has its own Kalmanson resistance matrix, where identified nodes are assigned 0 resistance between them. The associated circular split system will be a bulb of the cactus split system $\rw(N).$
 We get the following:

\begin{thm}
    \label{matchcaxt}
For a cactus network $N$, the map $\bx$ coincides with both the Kalmanson and the induced split systems of the planar dual. That is, $\bx(N) = \rw(N^*) = \sig(N^*) .$ Respectively, we have $\rw(N) = \sig(N) = \bx'(N^*).$ 
\end{thm}

\begin{proof}
 We prove the first triple inequality, and the other follows from the same treatment with appropriate clockwise and counterclockwise rotations.  A cactus network $N$ has bulbs separated by vertices which are the boundary nodes sharing several labels from $[n]$. No dual edge can connect nodes from separate bulbs. Thus the dual $N^*$ has connected components each consisting of nodes from the same bulb of $N$. Vice-versa, connected components of $N$ correspond to bulbs in the dual $N^*$. Thus we ask if  performing the operation of $\sig$ (or $\rw$) on a connected component of the dual $N^*$ has the same combinatorial result as performing the operation of $\bx$ on the corresponding bulb of $N.$ Indeed, the same arguments hold as for the case of a single bulb with a single connected component in Theorem~\ref{match}: edges in each of the bulb-wise Kron reductions $K(N)$  correspond to splits displayed by $N^*$.

We also need to check that the respective operations give matching values for the weights of the splits. This also follows from Theorem~\ref{match}, applied to the single bulb of $N$ and the corresponding connected component of $N^*$.
The weight of the split in the cactus split system $\bx(N)$ is the same as the edge conductance in the Kron reduction of a single bulb, which in turn is the weight of the split found by $\rw(N^*)$ and $\sig(N^*)$ as in the non-cactus case.

\end{proof}

 \begin{figure}[h]
\includegraphics[width=\textwidth]{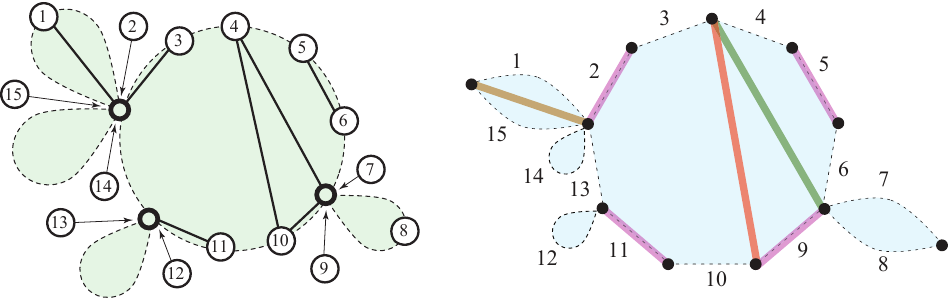}
\caption{The cactus split system $\sig(N) = \bx'(N^*)$ associated to the network from Figure~\ref{f:bubble}.}
\label{f:cactus-poly}
\end{figure}

\subsection{Summary of Maps}

Figure~\ref{coolmap} displays an overview all the maps discussed so far. We note that maps $\bx$ and $\bx'$ take as input any electrical network (regardless of planarity) and return a circular planar split system. However, these two maps rely on a given cyclic order of the boundary nodes, typically given as the clockwise order. In contrast, the map $\sig$ must take as input a circular planar electrical network $N$, and yet gives the same result if applied to an equivalent network with a different cyclic order. 
Finally, the map $\rw$ takes as input any electrical network $N$ --- Kalmanson, planar, or neither, allowing for imperfect data --- and returns a circular planar split system.  However this output is only guaranteed to have resistances corresponding to the original network in the case that $N$ has a Kalmanson resistance metric.

This construction  raises new questions. First, although here we restrict our study to planar networks with a given clockwise order of nodes, the map $\bx$ can operate on any network, as long as the nodes are given a clockwise order first. The results of that mapping for non-planar examples are interesting objects for future study. 

Secondly the map $\bx$ carries the structure of duality on electrical networks to a new duality on their image in the circular split systems. We note that the map $\bx$ is reminiscent of the $T$-duality map on decorated permutations \cite{hypersimp}. Moreover, $\bx$ carries an equivalence relationship on circular split systems --- under twisting as defined by Devadoss and Petti \cite{dev-petti} --- to a new combinatorial equivalence on response matrices which rewires the (possibly non-planar) electrical networks. What is the physical meaning of this equivalence?
Indeed, $\bx$ shows the poset and CW cell structures on circular split systems as a coarsening of those on circular planar networks. It also carries the cell structure of split systems to a new cell structure on all circular electrical networks. This new combinatorial structure is unexplored.

\begin{rem}
    In \cite{lam} is is stated that ``There is a natural notion of the response matrix of a cactus network: we specify
voltages at boundary vertices such that identified vertices  are
assigned the same voltage.'' 
This matrix $M(N)$ can be posited as the limit of the response matrix as some of the edges of a network are taken to zero resistance. 
Moreover, the response matrix of a cactus network will have entries of $\infty$ for the conductance between nodes that share a vertex, and corresponding diagonal entries of $-\infty.$ 
The corresponding resistance matrix will have entries of 0 for pairs of nodes that share a vertex of the cactus network, but the other entries can be calculated directly by isolating the two nodes and using Kirchoff and Ohm's laws to find effective resistance between them.

From Figure~\ref{fig:compmap} we can find a limiting $M(N)$ and $W(N)$ for the entire network. 
$$ M(N) = 
 \begin{bmatrix}
 -1 & 1/3 & 0 & 0 & 1/3 & 1/3\\
 1/3 & -\infty & 1/3 & 1/3 & \infty & \infty \\
 0 & 1/3 & -2 & 1 & 1/3 & 1/3\\
 0 & 1/3 & 1 & -2 & 1/3 & 1/3\\
 1/3 & \infty & 1/3 & 1/3 &-\infty & \infty \\
 1/3 & \infty & 1/3 & 1/3 &\infty & -\infty \\
 \end{bmatrix} 
\hspace{.4in} W(N) =  
 \begin{bmatrix}
0 & 1 & 5/3 & 5/3 & 1 & 1 \\
1 & 0 & 2/3 & 2/3 & 0 & 0 \\
5/3 & 2/3 & 0 & 2/3 & 2/3 & 2/3 \\
5/3  & 2/3 & 2/3 & 0 & 2/3 & 2/3 \\
1 & 0 & 2/3 & 2/3 & 0 & 0 \\
1 & 0 & 2/3 & 2/3 & 0 & 0 \\
 \end{bmatrix} $$

\end{rem}

%
%
\section{Enumeration of Species}\label{s:count}
\subsection{Lagrange Inversion}

Our graphical map $\bx$  takes any circular cactus network  to a cactus planar split system.  We leverage this map and show how to count these systems, all based on the following result:
  
\begin{thm}\label{countery}
For a given cyclic order on $[n]$, the number of unweighted cactus split systems is given by the OEIS sequence A136654. The generating function is given by
$$\Spp(x) \ = \ \left(\frac{1}{x}\right){\textup{Inverse}}\left(\frac{x}{\Sp(x)}\right), \ \ \ \textup{where} \ \  \ \Sp(x) \ = \ \sum_{k=0}^{\infty}2^{k \choose 2}x^k,$$
and ``{\textup{Inverse}}'' refers to taking the inverse function.
\end{thm}

\begin{rem}
To find the number of cactus split systems on $n,$ use the Lagrange inversion theorem on $x/\Sp(x)$: Find the series expansion for $(\Sp_{n}(x))^{n+1}$ (where $\Sp_{n}(x)$ is the partial series up to $x^{n}$), take the coefficient of $x^{n}$ in that expansion, and divide by $n+1$. For instance, when $n=4$, we see  
$$(1+ x + 2x^2+8x^3+64x^4)^5  \ = \ 1 + 5 x + 20 x^2 + 90 x^3 + 585 x^4 +O(x^5).$$
Thus dividing the coefficient of $x^4$ by 4+1 gives us $|\Spp_4| = 585/5 = 117$.  
\end{rem}

\begin{proof}
We start by associating a species $\Spp$ to the combinatorial structures of cactus split systems  \footnote{A great reference for this section is \cite{bergeron_book}}. Define $\Spp([n])$ to be the set of pairs made up of a cactus split system on $[n]$ and a permutation on $[n].$ The permutation can be thought of as a 1-1 mapping of the exterior labels in their cyclic order to a new set of numerical tags $\{1,\dots,n\}$. This allows us to consider $\Spp(x)$ as simultaneously the ordinary generating function of the numbers of structure (types) of cactus split systems and the exponential generating function of the species $\Spp.$  
 
 Next, any instance of this species on $[n]$ can be created with the following steps. First for any $k\le n$ we choose $k$ tags: any subset  $P_1\subset [n]$ of size $k$; and simultaneously we choose a size-$k$ starting subset $S_1$ of our original cyclic labels which is required to contain the least label 1. For instance, to make the structure in Figure~\ref{f:cactus-poly} we let $k=2$ and chose $S_1 = \{1,15\}.$   The tags are attached to the labels in $S_1$ (in any order) and a set of splits is assigned to $S_1.$ There are $2^{k \choose 2}$ possibilities for those splits. (The tags are not shown in Figure~\ref{f:cactus-poly} because we only wanted to display the structure type.)
 
Now, the set $S_1$ is a part of a noncrossing partition of the clockwise circular labels $[n]$, and the remaining labels not in $S_1$ are subdivided by the contiguous portions of $S_1$ into at most $k$ contiguous (in the cyclic order) sets. These contiguous sets are then used to continue the process: the species $\Spp$ is applied to each of them recursively using the least label as starting set, with all the possible partitions of the remaining $n-k$ tags. Each time, we draw the result as attaching the new cactus at the point on the old where the cyclic order of labels was cut.  Thus we see a functional equation for the species: 
$$\Spp \ =\ \sum_{k=0}^{\infty}\ 2^{k\choose 2}X^k\Spp^k\,,$$
or equivalently, $\Spp=\Sp\circ(X\cdot \Spp)$,  where $\Sp$ is the species of ordinary split systems and $X$ is the singleton species. That is, any instance of the species of cactus split systems can be described recursively as an ordinary split system of cactus split systems each with at least one element (which we tag), but possibly with more: the attached cactus at the point after that tagged label. Therefore, in terms of exponential generating functions, $\Spp(x)=\Sp(x\Spp(x))$ .   Multiplying both sides by $x$ and rearranging yields 
$$x \ = \ \frac{x\Spp(x)}{\Sp(x\Spp(x))}\,,$$ which implies 
$$x\Spp(x) \ = \ {\textup{Inverse}}\left(\frac{x}{\Sp(x)}\right)$$
and thus our result.
\end{proof}
 
\subsection{Corollaries}

The following is an immediate consequence:

\begin{corollary}\label{bigform}
Counting the cactus split systems for a given cyclic order on $[n]$ yields
$$|\Spp_n| \ = \ \frac{1}{n+1}\left(\ \sum_{j_0+\dots+j_n = n}\left(\ \prod_{i=0}^n2^{j_i \choose 2}\right)\right)\,,$$
where the sum is over ordered lists of non-negative integers summing to $n$.
\end{corollary}

Some numbers are shown in Table~\ref{topper}. For comparison, the numbers of unweighted equivalence classes of cactus networks for a given cyclic order are given by $(2n-1)!!$, which counts the perfect matchings on $[2n]$ . For instance, for a given cyclic order for $n=4$, there are  15 cactus split systems not corresponding to a cactus circular planar electrical network --- they correspond instead to non-planar cactus networks.  At the same time, there are 4 types of unweighted planar cactus networks on $[4]$ that all map to the same split system: the top dimensional cell plus 3 more planar cactus networks whose Kron reduction is the complete graph, as seen in Figure~\ref{fig:fours}. Thus the counts of cells compare: 117 - 15 + 3 = 105. 

Moreover, we can also count the cactus electrical networks, using the same arguments as in the proof of Theorem~\ref{countery}, or via a comment of P. Hanna in entry [A111088] of \cite{oeis}: 

\begin{corollary}
Counting the cactus electrical networks yields
$$|\Omm_n| \ = \ \frac{1}{n+1}\left(\ \sum_{j_0+\dots+j_n = n}\left(\ \prod_{i=0}^n |\Om_{j_i}| \right)\right) \ = \ (2n-1)!!\,,$$
where $|\Om_n|$ is the number of distinct classes of  non-compact electrical  networks, by Equation~\eqref{e:main}.
\end{corollary} 

\subsection{Plabic Tilings and Ranges}

Counting phylogenetic networks is an active area of research, with the application of bounding the search time for reconstruction algorithms \cite{newgamb, gambette-kelk}. An especially useful sort of phylogenetic split system is the induced split system $\sig(N)$, as described combinatorially by Gambette, Huber, and Scholz \cite{Gambette2017}. These are somewhat difficult to count, especially when we allow for multiple connected components, for two reasons. 
First,  the induced split system for networks with multiple connected components is a cactus split system, as first defined in this paper. Second, the fact that the induced map takes disconnected networks to compact split systems, and compact networks to disconnected split systems, means that it is hard to count the split systems in its range using a recursive species procedure, at first glance. 

However, our new graphical map $\bx$ doesn't have that problem: it preserves the connected components and the compactified (cactus) structure, so we can enumerate the items in the range of $\bx$ using our functional formula $A = B\circ(X\cdot A)$. Theorem~\ref{top} then tell us  that the graphical map and the induced map have the same overall range, so by counting one you count the other. Thus the first step is to report a method for counting the non-compact image of $\bx,$ so that we have input for our final count.

\begin{defn}\label{tau}
For an $n$-node circular planar electrical network $N$, the \emph{plabic tiling} $\tau(N)$ is a collection of polygons with edges labeled with $[n]$, each subdivided and shaded as follows: There is one polygon for each part of the partition making up the cographical split system $\bx'(N^*).$ Add any diagonal that is a bridge in $\bx'(N^*)$. Regions of $\bx'(N^*)$ with no diagonal edges are shaded in $\tau(N)$ whereas regions of $\tau(N)$ that correspond to complete graphs in $\bx'(N^*)$ remain unshaded. 
\end{defn}

\begin{exmp}
Figure~\ref{f:plabic} shows the construction of $\tau(N).$ We use $\bx'(N^*)$ from Figure~\ref{f:intro}, but note that the result is easy to see directly from observing the edges of $K(N^*).$
\end{exmp}

\begin{figure}[h]
\includegraphics[width=\textwidth]{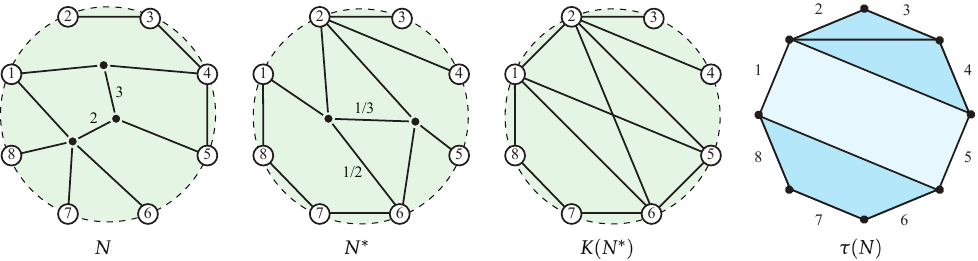}
\caption{Networks from Figure~\ref{f:strands} along with $K(N^*)$ and the plabic tiling $\tau(N$).}
\label{f:plabic}
\end{figure}

Curtis and Morrow \cite{curtisbook} show that a planar circular electrical network gives rise to a response matrix $M$ with non-negative circular minors. In particular, this means that for any pair of crossing diagonals in the Kron reduction, the four other diagonals using the endpoints of the two crossing diagonals must be present.  That feature guarantees that the Kron reduction will appear as non-overlapping cliques and empty polygonal regions, which we state as the following:

\begin{corollary}\label{c:obstruct}
If  a circular electrical network $N$ is planar, then the Kron reductions $K(N^*)$ and $K(N)$, and thus the split systems $\bx'(N^*)$ and $\bx(N)$, allow constructions of the plabic tilings $\tau(N^*)$ and $\tau(N)$.
\end{corollary}



From Theorem~\ref{top}, and noting that taking the dual is a bijective operation on the (clockwise ordered) cactus electrical networks $\Omm_n$, we conclude that the image of $\Omm_n$ under $\bx$ is the same as its image under $\sig$. We call that image the \emph{faithful} split systems as in \cite{frontiers}. Thus
for compact clockwise networks, the faithful (unweighted) split systems (that is, the unweighted induced systems) are counted by  $|\sig(\Omm_n)| = |\bx({\Omm}_n)|$. These cells  form a subcomplex of $\Spp_n.$  The $f$-vector of that subcomplex for $n=4$ is $( 14, 28, 28, 20, 9, 2, 1)$.

The total number of (noncompact) unweighted split systems in the image $\bx(\Om_n)$ in $\Sp_n$, which form a subcomplex of $\Sp_n,$
can be counted using a formula found in \cite{holm}.
There, the Ptolemy diagrams  are described as polygons with a subset of diagonals, obeying the rule that for every pair of crossing diagonals in a diagram, all the other diagonals using their four endpoints (edges of a quadrilateral) must be included in that diagram.\footnote{The Ptolemy diagrams correspond to torsion pairs in the cluster category of type $A_n$ \cite{holm}.}  This rule precisely describes the Kron reductions of circular planar electrical networks (in clockwise order), with the extra requirement that the boundary edges (between consecutive nodes on the boundary) all be included.  
The Ptolemy diagrams $\mathcal{P}_n$ of the $n$-gon for $n\ge 3$, and using $k = n-3$ are counted in \cite{holm} by
$$|\mathcal{P}_n| \ =  \ \frac{1}{k+2} \sum_{j= 0}^{\lfloor\frac{k+1}{2}\rfloor}\ 2^j{k+1+j \choose j}{2k+2 \choose k+1-2j}\,,$$ 


\noindent resulting in the OEIS sequence [A181517]. Note that the plabic tilings $\mathcal{T}_n$ of a single polygon with  $n$ sides are in direct bijection with the Ptolemy diagrams of size $n$, for $n\ge 3$. We also include in $\mathcal{T}_n$ the unique diagrams of size $n=2$, 1, and 0: a single edge (a trivial split), a single vertex, and the empty diagram. The numbers of these single polygon tilings are then $|\mathcal{T}_n| = 1,1,1,1,4,17, 82, 422,...$.

We wish to count the cells in the image $\bx(\Om_n).$ The only difference between the polygonal pictures of split systems in the image of $\bx$ and the Ptolemy diagrams is that the boundary edges (trivial splits) of the former can be included or excluded, when they are not part of a clique of size 4 or more. In terms of our plabic tilings of a polygon, the trivial splits are all boundary edges of a shaded region (or the single edge diagram with no bounded region). Denoting the number of these optional trivial splits for a given plabic tiling $s \in \mathcal{T}_n$ by $t(s)$ we have:
\begin{corollary}\label{countbx} The number of unweighted faithful circular split systems, which label the cells in the range of $\bx,$ is given by: 
$$|\bx(\Om_n)| \ = \ \sum_{s\in {\mathcal{T}}_n} 2^{t(s)}.$$
\end{corollary}

\begin{exmp}
When $n=4$, the four plabic tilings of a square are $\mathcal{T}_4=$\includegraphics[width=1.25in]{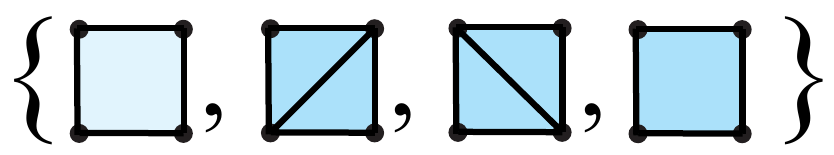}. Thus we have $|\bx(\Om_4)| = 2^0 + 2^4 +2^4 + 2^4 = 49,$ as seen in Figure~\ref{fig:fours}, inside the dashed line. For $n=0,\dots,6$ we have  $|\bx(\Om_n)| = 1, 1, 2, 8, 49, 373, 3196 $ respectively.  
\end{exmp}

  The proof of Theorem~\ref{countery}  can now be easily used to to count the induced images of the cells of cactus electrical networks inside the cactus circular split systems. Since finding the dual network is a bijection on the cells of $\Omm_n,$ then using our Theorem~\ref{top}, we have that the set of cells of cactus induced systems is the same as the set of cells of cactus graphical systems (images of  the graphical map). The important feature is that the cactus graphical systems as a species $A$ are organized precisely as $A = B \circ (X\cdot A)$ where $B$ is the species of ordinary graphical systems. Thus the proof of Theorem~\ref{countery} gives us:    

\begin{corollary}\label{last}
The enumeration of the images of cactus electrical networks yields
$$\ |\sig(\Omm_n)| \ = \ |\bx(\Omm_n)| \ =  \ \frac{1}{n+1}\left(\ \sum_{j_0+\dots+j_n = n}\left(\ \prod_{i=0}^n |\bx(\Om_{j_i})| \right)\right)\,.$$

\noindent which can be expanded to:

$$|\sig(\Omm_n)| \ = \ \frac{1}{n+1}\left(\ \sum_{j_0+\dots+j_n = n}\left(\ \prod_{i=0}^n \ \left(\ \sum_{s\in {\mathcal{T}}_{j_i}} \ 2^{t(s)}\right) \right)\right)\,,$$
where $j_0,\dots,j_n$ are ordered non-negative integers, ${\mathcal{T}}_{j_i}$ is the set of plabic tilings of a polygon with $j_i$ boundary edges, and $t(s)$ is the number of boundary edges of the shaded regions of $s$.
\end{corollary}

\noindent  Some numerical results are seen in Table~\ref{topper}. For example, when $n=3,$ we have the total: $(1/4)(4(8) + 12(2) + 4(1)) = 15,$ as pictured in Figure~\ref{uno}.
When $n=4,$ we have the total: $(1/5)(5(49)+20(8)+10(4) + 30(2) + 5(1)) = 102.$ We would like to find a more efficient formula, but will leave that as an open question.

\section{Global Compactified Spaces} \label{s:glob}
\subsection{}

With future work in mind, we need to extend the definition of the space of circular electrical networks and their compactifications to encompass all circular permutations of $[n]$. We call this extending to the \textit{global} case. Here we introduce the spaces, enumerate some cells for small dimensions, and prove species formulas. Maps of this paper extend to the global spaces : but only $\sig$ and $\rho$ are well defined on the new global equivalence classes.  The map $\bx$ requires all its inputs and output to have a fixed cyclic ordering of nodes, so it is only well defined on equivalence classes within a single chamber of the global space.

\begin{defn} A global compact circular planar electrical network $N$ is a circular (cactus) network with $n$ boundary nodes labeled by $[n]$, not necessarily in clockwise (nor any specific) order. 
\end{defn}

\begin{defn}
    Two global compact circular planar electrical network $N$ and $N'$ are equivalent if they both have the same set of pairwise effective resistances $W_{ij}$ between nodes labeled $i$ and $j.$  
\end{defn}

Note that for instance this means that if $N\equiv N'$ they must both have the same sets of identified nodes (where resistance is 0) and connected components, since resistance between disconnected nodes is $\infty.$ 
Indeed, two planar electrical networks with different cyclic orders $c$ and $c'$ of boundary nodes are still equivalent if they have the same response matrix with rows and columns ordered by the usual counting order of $[n].$ Notice that planarity must hold with respect to the cyclic orders: so the two equivalent networks are both planar with respect to both orders.

 For any fixed cyclic order $c$ of $[n]$ there is a copy of $\Om_n$ and of $\Omm_n$, called $\Om_n^c$ and $\Omm_n^c$, the circular networks with boundary nodes in the order $c$.  Thus the following:
\begin{defn}
   We define the space $\Om^{global}_n$
as the union of the $(n-1)!/2$ copies of $\Om^c_n$, called chambers, one for each cyclic order $c$; glued together along equivalent networks. We define the space $\Omm^{global}_n$
as the union of the $(n-1)!/2$ copies of $\Omm^c_n$, glued together along equivalent networks.
\end{defn}
 
The cells of $\Om^{global}_n$ are labeled with minimal circular planar  networks just as for $\Om_n$. The 0 and 1-dimensional cells are the same for both the global and clockwise spaces. For $n=4$, the extra cells of dimension 2--6 are seen in Figure~\ref{newbies}. These can be added to columns 3--5 in Figure~\ref{fig:fours} to find the $f$-vector  $(1, 6, 15, 20, 16, 9, 3)$ for $\Om^{global}_4$. The $f$-vector for $\Om^{global}_5$ is $(1, 10, 45, 120, 215, 288, 310, 255, 150, 60, 12)$.

\begin{figure}
    \centering\includegraphics[width=0.65\linewidth]{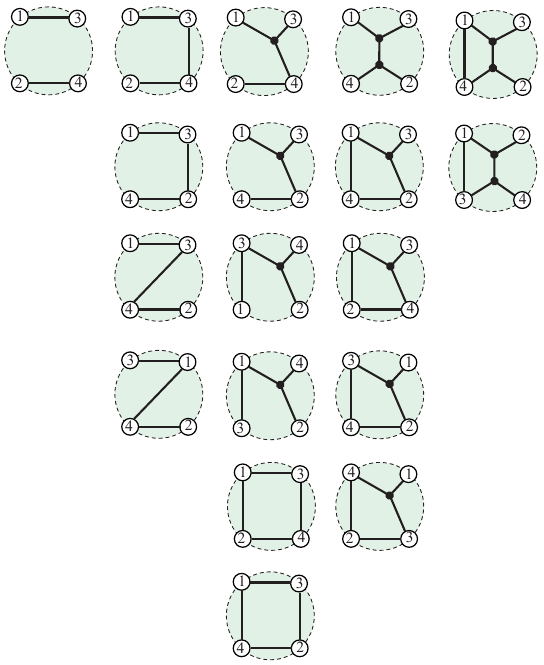}
    \caption{The global circular planar electrical networks of $\Om^{global}_4$ that cannot be realized in planar clockwise cyclic order.}
    \label{newbies}
\end{figure}

Similarly we define the global space of circular split systems and its compactification. As mentioned above, a (global) circular split system using any cyclic ordering is actually the sort of circular split system that is usually studied in the phylogenetics literature. Indeed, the Kalmanson (matrix) metrics are defined as having an existing cyclic order for which they obey the conditions, and thus correspond to a system of splits for which there exists a cyclic order of $[n]$ that allows that system to be drawn as a polygonal diagram. 

\begin{defn}
   Let $\Sp^{global}_n$ denote the space of Kalmanson metrics on $[n]$, which bijectively correspond to circular split systems. This space is comprised of  $(n-1)!/2$ chambers, one for each cyclic order, each a copy of $\Sp_n.$ 
\end{defn}

 In \cite{dev-petti} these systems are studied with the additional requirement that all trivial splits are present, which reflects a phylogenetic situation in which all taxa are temporally removed from common ancestors. That source points out that visually the diagrams of the split system can be twisted along splits that are either not crossed by another split or not present in the system---and still represent the same split system. The (known terms in) the sequence of total numbers of cells of $\Sp^{global}_n$ are 2, 8, 112, 6976, 1332224, $\dots$.  These are based on the hand-counted values in \cite{dev-petti} and \cite{terhorst}, each multiplied by $2^n$ to allow any subset of the trivial splits to be included in the split system. The general term is an open question.
 
 When we compactify the space of split systems with a given cyclic order we represented the new points as non-crossing partitions of $[n]$, with the splits on each part of the partition. Now however we allow any 
 partition of $[n]$, and any cyclic order for the polygon on each part of the partition. 

 \begin{defn}
   Let $\Spp^{global}_n$ denote the space of Kalmanson metrics on parts of partitions of $[n]$, which correspond bijectively to circular split systems on each part of each partition. This space is comprised of  $(n-1)!/2$ chambers, one for each cyclic order, each a copy of $\Spp_n.$ 
\end{defn}
\begin{table}[h]
    \centering
    \begin{tabular}{|c|c|c|c|c|}
    \hline
      Space   &  number of 0-cells & total cells & $f$-vector & $\chi$ \\
      \hline\hline
 \rule{0pt}{2.6ex}\rule[-1.2ex]{0pt}{0pt} 
      $\Om_2 = \Om_2^{global} = \Sp_2 =\Sp^{global}_2$ & 1 & 2 & (1, 1) & 0\\
      \hline
 \rule{0pt}{2.6ex}\rule[-1.2ex]{0pt}{0pt}     
      $\Omm_2 = \Omm_2^{global} =  \Spp_2 = \Spp_2^{global}$ & 2 & 3 & (2, 1) & 1\\
      \hline
 \rule{0pt}{2.6ex}\rule[-1.2ex]{0pt}{0pt} 
      $\Om_3 = \Om_3^{global} = \Sp_3 = \Sp^{global}_3$ & 1 & 8 & (1, 3, 3, 1) &0\\
      \hline
 \rule{0pt}{2.6ex}\rule[-1.2ex]{0pt}{0pt}      
      $\Omm_3 = \Omm_3^{global} =  \Spp_3 = \Spp^{global}_3$ & 5 & 15 & (5, 6, 3, 1) & 1\\
      \hline
 \rule{0pt}{2.6ex}\rule[-1.2ex]{0pt}{0pt}  
      $\Om_4$ & 1 & 52 [A111088] & (1, 6, 14, 16, 10, 4, 1) & 0\\
     \hline
     \rule{0pt}{2.6ex}\rule[-1.2ex]{0pt}{0pt}  
      $\bx(\Om_4)$ & 1 & 49 [Cor.~\ref{countbx}] & (1, 6, 14, 16, 9, 2, 1) & 1\\
     \hline
\rule{0pt}{2.6ex}\rule[-1.2ex]{0pt}{0pt} 
       $\Sp_4$ & 1 & 64 [A006125] & (1, 6, 15, 20, 15, 6, 1) & 0\\
      \hline
 \rule{0pt}{2.6ex}\rule[-1.2ex]{0pt}{0pt}  
      $\Omm_4$ & 14 [A000108] & 105 [A001147] & (14, 28, 28, 20, 10, 4, 1) & 1\\
      \hline
\rule{0pt}{2.6ex}\rule[-1.2ex]{0pt}{0pt} 
       $\bx(\Omm_4) = \sig(\Omm_4)$ & 14 [A000108] & 102 [A136654] & (14, 28, 28, 20, 9, 2, 1) & 2\\
      \hline
 \rule{0pt}{2.6ex}\rule[-1.2ex]{0pt}{0pt} 
       $\Spp_4$ & 14 [A000108] & 117 [A136654] & (14, 28, 29, 24, 15, 6, 1) & 1\\
      \hline
 \rule{0pt}{2.6ex}\rule[-1.2ex]{0pt}{0pt} 
      $\Om^{global}_4$ & 1 & 70 [open] & (1, 6, 15, 20, 16, 9, 3) & 0\\
      \hline
 \rule{0pt}{2.6ex}\rule[-1.2ex]{0pt}{0pt}  
      $\Sp_4^{global}$ & 1 & 112 [open] &  (1, 7, 21, 34, 31, 15, 3) & 0 \\
     \hline
 \rule{0pt}{2.6ex}\rule[-1.2ex]{0pt}{0pt} 
      $\Omm^{global}_4$ & 15 [A000110] & 133 [open] & (15, 31, 33, 26, 16, 9, 3) & 1\\
      \hline
 \rule{0pt}{2.6ex}\rule[-1.2ex]{0pt}{0pt} 
      $\sig(\Omm^{global}_4)$& 15 [A000110] & 124 [open] & (15, 31, 33, 26, 13, 3, 3) & 4\\
      \hline
 \rule{0pt}{2.6ex}\rule[-1.2ex]{0pt}{0pt} 
      $\Spp_4^{global}$ & 15 [A000110] & 169 [open] & (15, 31, 36, 38, 31, 15, 3) & 1\\
      \hline
      \hline
    \end{tabular}
    \vspace{.15in}
    \caption{CW-complexes in this paper: for $n=2,3,4$, some sequences, $f$-vectors, and Euler characteristic $\chi$.}
    \label{ftabglob}
\end{table}
Unweighted cactus networks with any cyclic order correspond to the cells of the resulting CW-complex structure on $\Omm^{global}_n$.  
The 1-skeletons of $\Omm^{global}_n$ and $\Spp^{global}_n$ are made of 0-cells that have no edge (respectively no split) and 1-cells that have a single edge (single split). Thus these are identical as graphs. Figure~\ref{fig:skelly4globOmm} shows the 1-skeleton of $\Omm^{global}_4$. Figure~\ref{fig:skelly4globSpp} shows the 1-skeleton of $\Spp^{global}_4$, and is drawn as the mirror image of Figure~\ref{fig:skelly4globOmm}: the map $\sig$ takes electrical networks in Figure~\ref{fig:skelly4globOmm} to the corresponding split system in the mirror image.  Of course the number of higher dimensional cells in $\Omm^{global}_n$ is generally fewer than in $\Spp^{global}_n$ since planarity of the circuit is required, as opposed to just planarity of the splits. The $f$-vector for $\Omm^{global}_4$ is $(15, 31, 33, 26, 16, 9, 3)$ with 133 total cells. The $f$-vector for $\Omm^{global}_5$ is $(52, 160, 270, 345, 375, 378, 340, 255, 150, 60, 12)$ with 2397 total cells.

The general numbers of cells for these global spaces, counted here by hand, is an open question, and the first step in an exciting new study of the topology of the spaces themselves. The Euler characteristics (alternating sums of the $f$-vectors we  just listed) are both 1, so we conjecture that the global compactified spaces are contractible. Table~\ref{top2} lists the total numbers of cells in global spaces that we have so far. Finding a formula for any of these is an open question. In Table~\ref{ftabglob} we list some face statistics of the examples of spaces in this paper, with OEIS entry names for the sequences that have proven formulas for all $n$. Certain sequences begin with the same terms, since the spaces coincide (and have only one chamber) for $n \le 3$.

 \begin{figure}[h!]
    \centering
    \includegraphics[width=\linewidth]{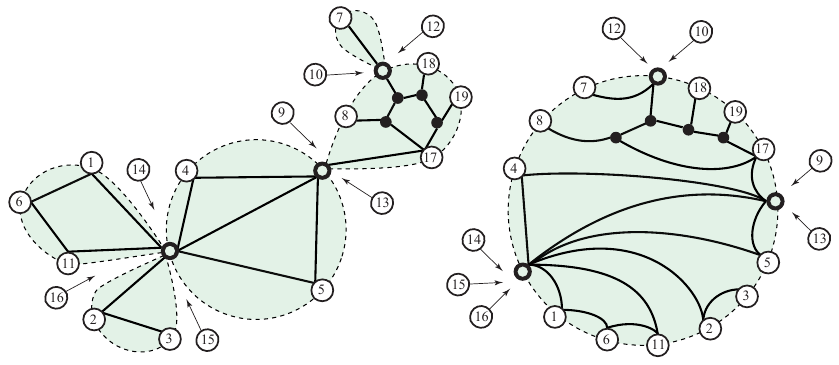}
    \caption{Equivalent global cactus networks.}
    \label{simp}
\end{figure}

\subsection{Global species}

 The sequence of numbers of 0-cells for $\Omm^c_n$ is shown to be the Catalan numbers in \cite{lam2}. The sequence of numbers of 0-cells for $\Spp_n$ is also given by the Catalan numbers, since the 0-cells correspond to non-crossing partitions of $[n].$  The sequence of numbers of 0-cells for $\Spp^{global}_n$ is then clearly the Bell numbers, since those give the numbers of partitions of $n$. That in turn implies that the Bell numbers also count the numbers of 0-cells in  $\Omm_n^{global}.$ In fact, the general cells of both species can be described as composition of species.

 \begin{thm}\label{globspecs}
   The species of global cactus split systems obeys:
     $$\Spp^{global} = E_+\circ \Sp^{global}$$ where $E_+$ is the species of non-empty sets.
   The species of global cactus electrical networks obeys $$\Omm^{global} = \Om^{global} \circ E_+.$$
 \end{thm}

  \noindent That is, a global cactus split system is formed by partitioning the set $[n]$ and then making a global ordinary split system on each  part. As  exponential generating functions we have: 

\begin{large}
  $$  {\Spp^{global}(x) = e^{\Sp^{global}(x)}-1} $$
\end{large}
\noindent We of course only know some of the function in the exponent, but this allows us to check our numbers for small $n.$ We fill in the portion known and find the Taylor series:
\vspace{.1in}

\begin{center}
$\Spp^{global}(x)=$ \begin{Large}
$e^{(x+2x^2/2+8x^3/6+112x^4/24+6976x^5/120+1332224x^6/720 + \dots)}$
\end{Large}$-1$
$\,\,\,\,\,\,\,\,\,\,\,\,\,\,\,\,\,\,\,\,\,\,\,=x + (3 x^2)/2 + (15 x^3)/6 + (169 x^4)/24 + (7857 x^5)/120 + (1381211 x^6)/720 +\dots $
\end{center}
\noindent as seen in Table~\ref{top2}. In contrast, a global cactus electrical network is formed by first partitioning $[n]$ and then making a global ordinary electrical network with the parts of that partition as nodes. 
As  exponential generating functions we have: 

\begin{large}
$$\Omm^{global}(x) = \Om^{global}(e^x-1)$$
\end{large}
\noindent Again we fill in the portion known and find its Taylor series:
$$\Omm^{global}(x)=(e^x-1)+(e^x-1)^2/2+8(e^x-1)^3/6+70(e^x-1)^4/24+1466(e^x-1)^5/120 +\dots $$$$
= x + (3 x^2)/2 + (15 x^3)/6 + (133 x^4)/24 + (2397 x^5)/120 +\dots$$
\noindent also as seen in Table~\ref{top2}.
\begin{proof} [Proof of Theorem~\ref{globspecs}]
    For global cactus split systems, the only further explanation needed is due to the fact that our drawings of these global cactus split systems often show the parts attached at vertices, which reflects the process of compactification. However, the same sets of splits are shown regardless of how the parts are attached; in fact we could simply list the set of split systems on the parts, no attachment needed. 

    The global cactus electrical networks are also drawn in a way that disguises their simplest form. Once we describe that form, it will be clear that  a global cactus electrical network is formed by first partitioning $[n]$ into the sets of identified nodes, and then making a global ordinary electrical network on those identified nodes. The key is that with the freedom to put nodes in any circular order, a global cactus network can always be drawn as a simple circular planar network (with the identified nodes on its boundary). To do that, we choose any bulb to be central. Then we shrink and rotate (in either direction) its immediately adjacent attached bulbs inside its boundary, not crossing any edges. The nodes of the immediately attached bulbs can be made to coincide with the boundary of the central bulb. The process is repeated on secondary bulbs (other bulbs of the central bulb) recursively. See Figure~\ref{simp} for an example. 
\end{proof}

\begin{figure}
 \centering
 \includegraphics[width=\textwidth]{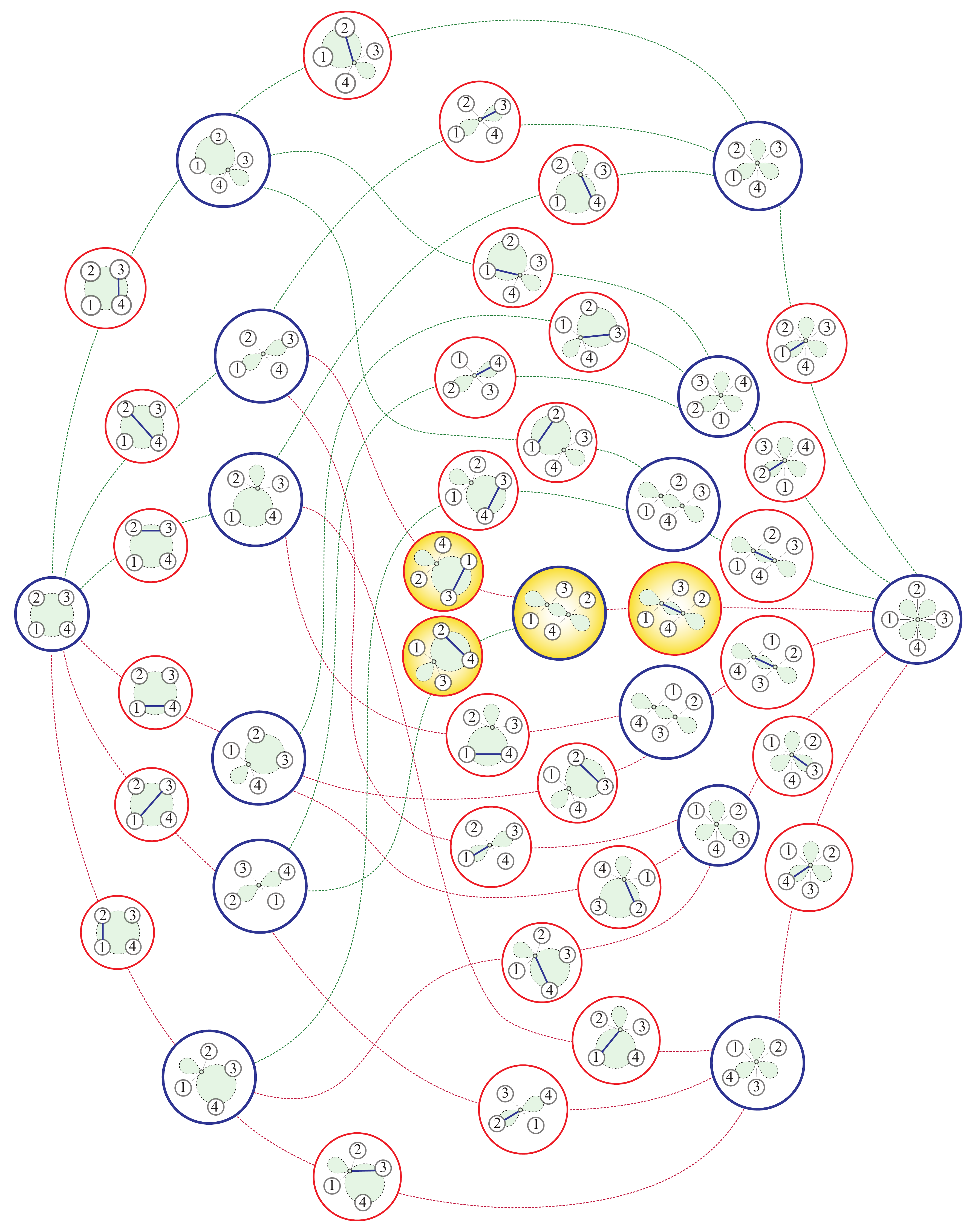}
 \caption{The 1-skeleton of $\Omm^{global}_4$. The highlighted networks in the center are the four which are excluded from Figure~\ref{fig:skelly4Omm}.}
 \label{fig:skelly4globOmm}
\end{figure}

\begin{figure}
 \centering
 \includegraphics[width=\textwidth]{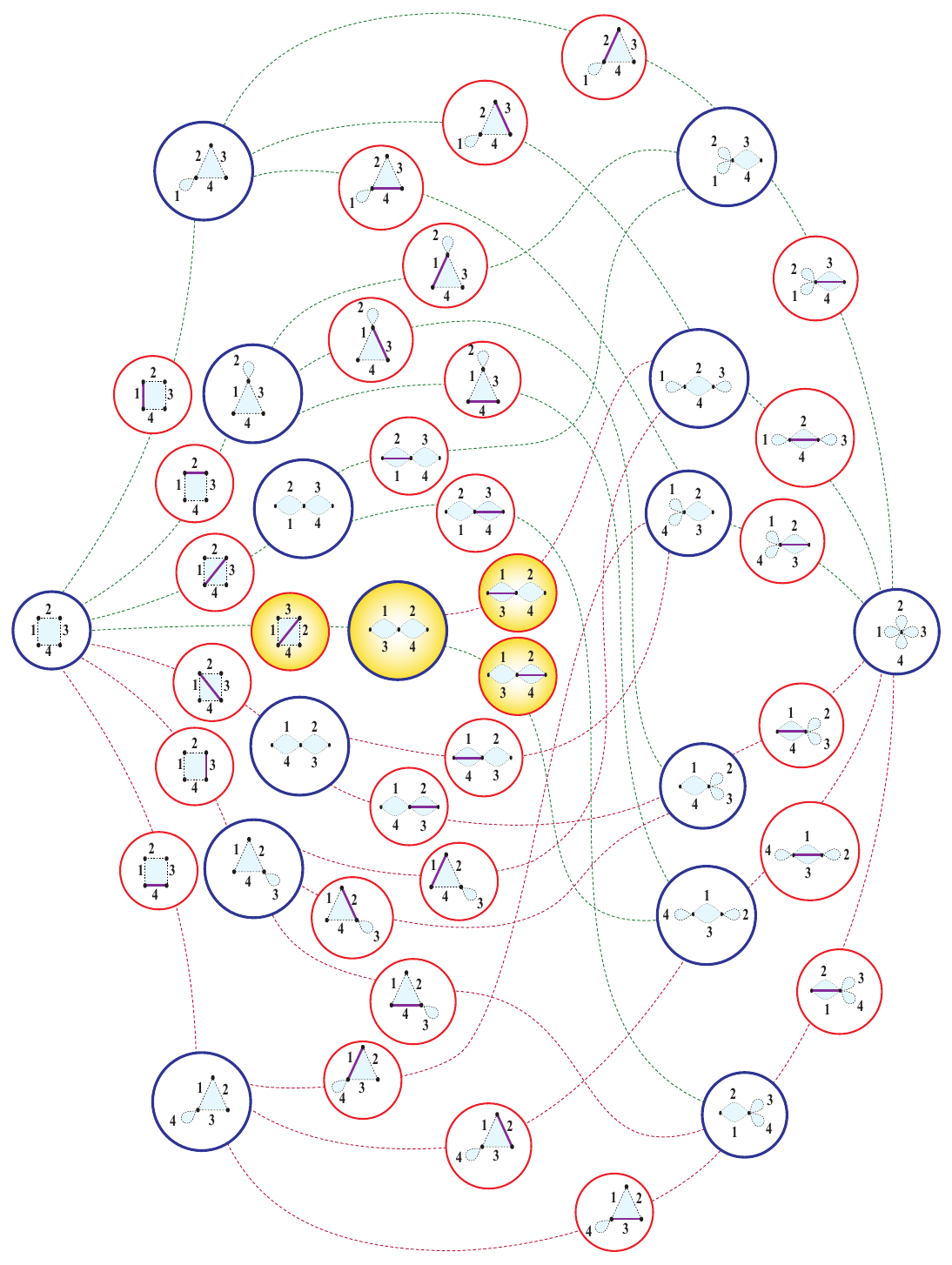}
 \caption{The 1-skeleton of $\Spp^{global}_4$. The highlighted networks in the center are the four which are excluded from Figure~\ref{fig:skelly4}.}
 \label{fig:skelly4globSpp}
\end{figure}

%
%
\section{Consequences} \label{s:corro}
\subsection{Corollaries}

Here we see several implications of the main theorems. Finding the split system $\rho(N) = \sig(N)$ in Theorem~\ref{bigth} is our contribution to Question (1) from the Introduction. In \cite{forc-pre}, we showed that the large-scale, galled-tree, structure of the network $N$ is captured by that of the split system $\rho(N)$.  The algorithm Neighbor-net takes any resistance matrix (perhaps of a circular electrical network), without regard for cyclic order, and returns a split network that exhibits a cyclic order and which reproduces precisely the tree-like portions of the circular planar electrical network. Recall that a \emph{bridge} is a split that is drawn as a diagonal that does not cross any others. The following corollary to Theorem~\ref{bigth} is proven in \cite{forc-pre}:

\begin{corollary}\label{bridg1}
The bridges and cut vertices of $N$ all become bridges in $\rho(N).$ Specifically, the bridges of $N$ are found among the bridges of $\rho(N).$
\end{corollary}

\noindent This corollary allows the problem of reconstructing a network from its response matrix to be subdivided into constructing sub-networks that are separated by bridges. The bridges are made visible in the split system. This preservation of bridges allows the modular reconstruction of $N$ from its response matrix, as demonstrated in \cite{forc-pre}. We next discuss the question of which cyclic orders can be consistent with a given network. 
\begin{defn}
For an $n$-node circular planar electrical network $N$, a cyclic order $c$ of $[n]$ is \emph{consistent} with $N$ when a new network can be created with that cyclic order of terminal nodes, preserving both the planarity and the pairwise resistances of $N.$
The set of consistent cyclic orders is denoted $\mathcal{O}_c(N)$.
\end{defn}

Just looking at a diagram of a circular electrical network with crossing wires, it can be hard to determine whether it is in fact planar. One nice feature of finding its split system $\rho(N) = \bx'(N*)$ is that if $N$ is actually planar, then its set of cyclic orders will be visible in the picture of the split system.  That visibility follows from the fact that the split system is unique, and that two pictures of the split system (with different cyclic orders) can always be transformed into each other via the twists along diagonals described in \cite{dev-petti}.  We make this more precise with our newly defined plabic tiling: $\tau(N)$ is thus useful for enumerating the consistent cyclic orders $\mathcal{O}_c(N)$, that is the number of top-dimensional cells which contain the cell of $N.$ We show that the number of cyclic orders of $[n]$ consistent with $N$ is found by considering the sizes of the regions of $\tau(N),$ answering Question (3) from the Introduction.

\begin{theorem}\label{counter}
 The number of cyclic orders consistent with a given connected, single-bulbed, circular planar electrical network $N$ is found by one half the product of factors, one factor for each region of $\tau(N).$ Each shaded region in $\tau(N)$ with $k$ sides contributes a factor of $(k-1)!$ and each unshaded region contributes a factor of 2.
\end{theorem}

\begin{proof}
    Since $N$ is planar, $\tau(N)$ will reflect the structure of the unique split system $\rho(N),$ which we can draw as a split network. Shaded regions correspond to cut-nodes of the split network: $k$ portions of the network meet at that cut node and thus can be arranged in any cyclic order, contributing $(k-1)!$ options. The unshaded regions can be flipped over, so each contributes 2 more independent options. Finally choosing various of these options can eventually flip over the entire network which is really the same cyclic ordering, so we need to divide by 2 in the end.  
\end{proof}

\begin{exmp}
The plabic tiling $\tau(N)$ in Figure~\ref{f:plabic} has 3 shaded regions contributing factors of $(3-1)!, (3-1)!$ and $(4-1)!$ respectively, and one unshaded region contributing a factor of 2. By Theorem~\ref{counter}, the number of consistent cyclic orders for the network $N$ from Figure~\ref{f:intro} is $|\mathcal{O}_c(N)| =2(2!)(2!)(3!)/2 = 24.$
\end{exmp}

The cases of a disconnected or cactus network with more than one bulb are more complicated, and we leave those for future work. 
For now we mention a  consequence that answers a question which is perhaps already familiar to circuit designers: Rearranging the terminal nodes of  a planar network might cause wires to cross. And if that new cyclic order is not consistent with the split system, it will not allow the new arrangement to be equivalent to a planar one. 

\begin{corollary}\label{badt}
Consider $N'$ formed by rearranging the boundary nodes of $N$ while keeping all edges intact, but in a cyclic order $c'$ not consistent with $\rho(N)$. Then $N'$ with its new cyclic order $c'$ is guaranteed to be nonplanar.
\end{corollary} 

This is exemplified in Figure~\ref{bad}, where $N'$ is obtained from $N$ by transposing terminals $5$ and $6$.
To see that the resulting $N'$ is non-planar, we can check the absence of its cyclic order from the list of 24 consistent circular
orders $\mathcal{O}_c(N).$ Notice that the graphs of $K(N)$ and $K(N')$ are the same shape; but the response matrices using the two orders to order their rows  would have differing values, and their minors would reflect the planarity and non-planarity respectively. Corollary~\ref{badt} simply allows that knowledge without directly calculating the minors.
\begin{figure}[h]
\includegraphics[width=\textwidth]{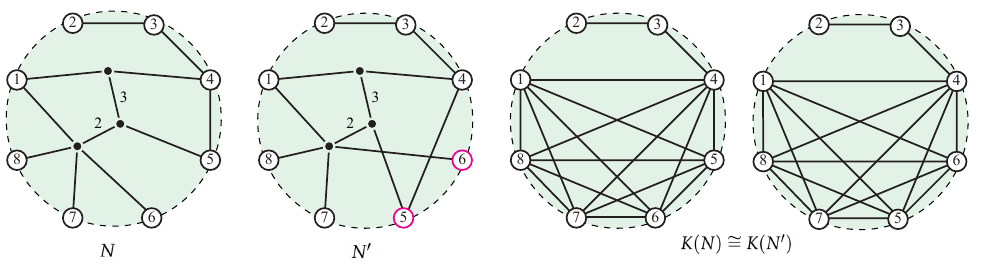}
\caption{The graph of $N'$ cannot be planar, despite sharing $K(N) \cong K(N')$. }
\label{bad}
\end{figure}


\noindent Corollary~\ref{c:obstruct} points out that the construction of $\tau(N)$ is always possible when $N$ is planar.
That fact allows our answer to Question (2) in the Introduction: there are certain obstructions to planarity easily visible in the Kron reduction.

\begin{exmp}
Figure~\ref{obs} shows two networks $N'$ and $N''$ obtained by adding a wire to network $N$ from Figure~\ref{f:intro}. The graph of $K(N')$ cannot be transformed to a plabic tiling, and therefore $N'$ cannot be planar, from Corollary~\ref{c:obstruct}. On the other hand, $N''$ and $K(N'')$ show that the converse is not always true.
\end{exmp}

\begin{figure}[h]
\includegraphics[width=\textwidth]{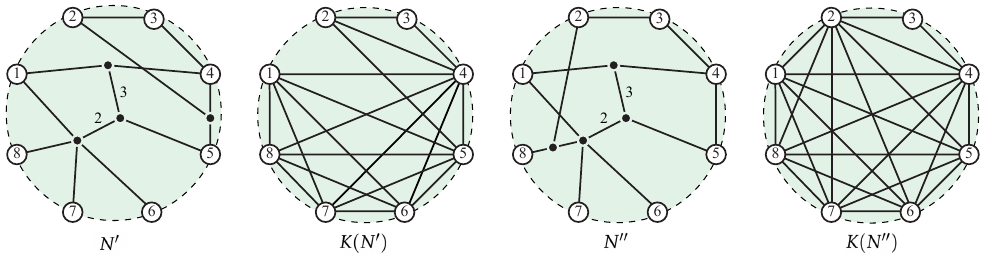}
\caption{Altering a network $N$ by adding a wire to get $N'$ and $N''$. }
\label{obs}
\end{figure}

\subsection{Traveling Salesman Polytopes}

 Since the Kalmanson and cographical split systems coincide, we can transfer known results relating split systems and the Symmetric Traveling Salesman Polytope, to the new arena of circular planar electrical networks. 

 \begin{defn}\label{face}
 For a cyclic order $c$ of $[n]$ let the incidence vector $\mathbf{x}(c)$ have components for each pair $i < j \in [n]$; equaling 1 if $i,j$ are adjacent in $c$, 0 if not. The ordering of the components is lexicographic. These incidence vectors for all cyclic orders make up the $(n-1)!/2$ vertices of the Symmetric Traveling Salesman Polytope STSP($n$).
\end{defn}

 \begin{defn}
 For a given circular planar electrical network $N$ define by $W_N(\mathbf{x}) = W_N\cdot \mathbf{x}$ the linear functional where the dot product is taken with the upper triangle of $W(N)$, not including the diagonal, read by rows.
 \end{defn}

\begin{theorem}\label{stsp}
 For a circular planar electrical network $N$, the linear functional $W_N$ is minimized simultaneously over the  STSP($n$) at the vertices $\mathbf{x}(c)$ corresponding to the orders in $\mathcal{O}_c(N)$ The value of that minimum is twice the sum of the splits of $\rho(N).$ 
\end{theorem}

\begin{proof}
This follows from Theorem 5.6 of \cite{scalzo} and Theorem 4.5 of \cite{frontiers}. The key is that any circuit of the boundary nodes in consistent cyclic order will accumulate two copies of every split. 
\end{proof}

\begin{exmp}
Theorem~\ref{stsp} shows that the set of incidence vectors corresponding to the cyclic orders in $\mathcal{O}_c(N)$ make up a face of STSP($n$). The 24 vectors associated to $N$ in Figure~\ref{f:intro} are the vertices of the face of STSP(8) with $f$-vector 
$(1, 24, 96, 186, 210, 145, 60, 13, 1)$.
For $N$ in Figure~\ref{f:intro}, the minimum is located simultaneously at the 24 vertices of STSP(8) where it takes the value 383/28. Thus, we answer Question (4) in the affirmative: By minimizing a given $W$, we can list the full set of cyclic orders allowing a planar construction for a circuit that gives $W$ upon measurement. Note that this listing is accomplished without knowing the actual structure of $N$, only its boundary measurements.
\end{exmp}



\newpage

\bibliographystyle{plain}
\bibliography{phylokron}{}

\end{document}